\documentclass[10pt,a4paper]{article}
\usepackage{amssymb,amsmath,amsfonts}
\usepackage{theorem}
\usepackage{enumerate}
\usepackage{bbold}
\usepackage{graphicx}
\usepackage[T1]{fontenc}
\usepackage[latin1]{inputenc}
\usepackage{graphicx}
\usepackage[left=2cm,right=2cm,top=2cm,bottom=2cm]{geometry}
\usepackage{color}
\usepackage{hyperref}
\usepackage{authblk}
\usepackage{accents}

\newenvironment{proof}[1][Proof]{\begin{trivlist}
\item[\hskip \labelsep {\bfseries #1}]}{\end{trivlist}}

\newtheorem{theorem}{Theorem}[section]
\newtheorem{lemma}[theorem]{Lemma}
\newtheorem{proposition}[theorem]{Proposition}

\newtheorem{remark}[theorem]{Remark}

\numberwithin{equation}{section}
\numberwithin{theorem}{section}

\newcommand{\ex }{\mathbb{E}}
\newcommand{\be}{\begin{equation}}
\newcommand{\ee}{\end{equation}}
\newcommand{\Prob}{\mathbb{P}}
\newcommand{\R}{\mathbb{R}}
\newcommand{\esp}{\mathbb{E}}
\newcommand{\eps}{\varepsilon}
\newcommand{\Holder}{H\"{o}lder }
\newcommand*{\qedBl}{\hfill\ensuremath{\blacksquare}}%

\title{
On small--noise equations with degenerate limiting system arising from volatility models
\thanks{\emph{date}: \today.
Corresponding authors: giovanniconfort@gmail.com, demarco@cmap.polytechnique.fr
\newline \indent \emph{Key words and phrases}: pathwise large deviations, square-root diffusions, tail asymptotics.
}}
\author{G. Conforti, S. De Marco, J--D. Deuschel
\\
Universit\"{a}t Potsdam, Ecole Polytechnique, TU-Berlin}
\renewcommand\footnotemark{}
\date{\today}

\begin{document}

\maketitle

\begin{abstract}
\noindent
The one-dimensional SDE with non Lipschitz diffusion coefficient
\begin{equation}\label{diffusion}
dX_{t} = b(X_{t})dt + \sigma X_{t}^{\gamma} dB_{t}, \quad X_{0}=x, \quad \gamma<1
\end{equation}
is widely studied in mathematical finance.
Several works have proposed asymptotic analysis of densities and implied volatilities in models involving instances of \eqref{diffusion}, based on a careful implementation of saddle-point methods and (essentially) the explicit knowledge of Fourier transforms.
Recent research on tail asymptotics for heat kernels \cite{DFJV:II} suggests to work with the rescaled variable
$ X^{\varepsilon}:=\varepsilon^{1/(1-\gamma)} X$ :
while allowing to turn a space asymptotic problem into a small-$\varepsilon$ problem with fixed terminal point,
the process $X^{\varepsilon}$ satisfies a SDE in Wentzell--Freidlin form (i.e. with driving noise $\varepsilon dB$).
We prove a pathwise large deviation principle for the process $X^{\varepsilon}$ as $\varepsilon \to 0$.
As it will become clear, the limiting ODE governing the large deviations admits infinitely many solutions, a non-standard situation in the Wentzell--Freidlin theory.
As for applications, the $\eps$-scaling allows to derive leading order asymptotics for path functionals of the process: while on the one hand the resulting formulae are confirmed by the CIR-CEV benchmarks, on the other hand the large deviation approach (i) applies to equations with a more general drift term amd (ii) potentially opens the way to heat kernel analysis for higher-dimensional diffusions involving \eqref{diffusion} as a component.
\end{abstract}

\section{Introduction}
The Wentzell--Freidlin large deviation theory studies the asymptotic behavior of the distribution on path space of the solution to the equation $dX^{\eps}_t = b(X^{\eps})dt + \eps \sigma(X^{\eps}_{t})dB_{t}, \ X^{\eps}_{0}=x$ as $\eps \to 0$, where $B$ is a Brownian motion.
When the coefficients $b$ and $\sigma$ are, say, Lipschitz functions, it is easy to see (with an application of Gronwall's Lemma) that the trajectories of $X^{\eps}$ converge in law to the deterministic solution of the ordinary differential equation 
$
d\varphi_t=b(\varphi_t)dt, \varphi_0=x.
$
The theory of large deviations accounts for the rate of this convergence: denoting $\varphi(h)$ the unique solution of the ODE
$
d\varphi_t=b(\varphi_t)dt +\sigma(\varphi_t)dh_t, \varphi_0=x
$
controlled by an absolutely continuous path $h$ with square integrable derivative $\dot{h}$, then the large deviation principle (LDP)
\[
W(X^{\eps} \in \Gamma ) \approx e^{-\frac1{\eps^2} \inf_{\phi \in \Gamma} I(\phi)}
\]
holds for subsets $\Gamma$ of $C([0,T])$, where $W$ stands for the Wiener measure.\footnote{The precise statement here is $-\inf_{\phi \in \accentset{\circ}{\Gamma}} I(\phi)
\le
\liminf_{\eps \to 0 }\eps^2 \log W(X^{\eps} \in \Gamma )
\le
\limsup_{\eps \to 0 }\eps^2 \log W(X^{\eps} \in \Gamma )
\le -\inf_{\phi \in \overline{\Gamma}} I(\phi)$.}
The rate function $I$ is given by $I(\phi)=\frac12|\dot{h}|^2_{L^2}$, where $h$ is the control steering the trajectory of the deterministic system along the given path $\phi$, that is $\varphi(h)=\phi$.
When the diffusion coefficient $\sigma$ is invertible, the control $h$ is identified by $\dot{h}_t=\sigma(\varphi_t)^{-1} (\dot{\varphi}_t-b(\varphi_t))$, yielding the typical form of the rate function \[
I(\phi)=\frac{1}{2}\int_{0}^{T} \frac{(\dot{\phi}_{t} - b(\phi_{t}))^{2}}{\sigma(\phi_{t})^{2}}dt.
\]

The intuition behind such a result is that we can write $X^{\varepsilon} (\omega) = X \left( \varepsilon \omega \right)$, where $X$ is the `pathwise' solution of $dX = b(X)dt +  \sigma(X) dB, X_{0} = x$.
If we accept that such a map $X$ exists and is regular enough, then the contraction principle in conjunction with Schilder's theorem for large deviations of Brownian paths \cite[Chap 1]{DEUSCHEL} provides the LDP and the rate function for $X^{\eps}$.
The standard assumptions under which such a program is carried are conditions of global Lipschitz continuity and ellipticity for the coefficients, see \cite{DEMBO,DEUSCHEL}.
Several works have aimed at weakening these assumptions and extending the class of equations for which the LDP holds. Dependence on $\varepsilon$ in both the drift and the starting point can be introduced, and global Lipschitz continuity can be replaced with (essentially) local Lipschitz-continuity and conditions for the non explosion of the solution (building on the idea of Azencott \cite{AZENCOTT} to exploit the quasi-continuity property of the It\^o map, that only relies on local properties of the equation coefficients).
We refer to \cite{BALDI} for a nice recent summary of sets of conditions under which the Wentzell--Freidlin estimate holds.

Recent research on heat kernel asymptotics \cite{DFJV:II} focuses on the tail behavior for correlated stochastic volatility models.
Exploiting the space-scaling properties of the log-price process $Y_t$ in some parametric models (namely: there exists $\theta>0$ such that the rescaled variable $Y^{\eps}_t := \eps^{\theta} Y_t$ has the same law as the log-price in a stochastic volatility model with driving noise $\eps dB_t$), the approach of \cite{DFJV:II} is to convert the asymptotic problem for the tail distribution, $W(Y_t>R)$ as $R \to \infty$, to the problem of small-noise probabilities, $W(Y^{\eps}_t>1)$ as $\eps \to 0$.
Then, a large deviation principle for the rescaled process serves as a building block to study the asymptotic behavior of the corresponding heat kernel (using the tools of Malliavin calculs and the Laplace method on path space, see \cite{Bism, BenA}).
This approach can be fully justified, and explicit computations are possible, for the stochastic volatility model of Stein--Stein \cite{SteinStein} (also known as Sch{\"o}bel--Zhu \cite{SchZhu} in the correlated case), where the stochastic volatility follows an Ornstein--Uhlenbeck process with constant diffusion coefficient, which is the main case-study of \cite{DFJV:II}.
As pointed out in \cite[Section 5.3]{DFJV:II}, in the framework of models where the volatility has square-root diffusion coefficient (main example: Heston), or more generally a diffusion coefficient of the form $x^{\gamma}$, $\gamma<1$ (as in \cite{AP} and \cite{LionsMus}), such a space-scaling approach leads to a situation where the same approach is not justified anymore (and a formal application of the resulting expansion even leads to a wrong conclusion).
As from \cite[Section 5.3]{DFJV:II}, ``\emph{curiously then even a large deviation principle for (the rescaled volatility process) as given above presently lacks justification}''.

To be more specific, consider the equation $dX_t = (\alpha + \beta X_t)dt + \sigma X_t^{\gamma} dB_t$ with positive initial condition $X_0=x>0$.
Looking for a value of $\theta$ such that $\eps^{\theta} X$ satisfies an equation with small-noise $\eps$ leads to define the rescaled process $X^{\varepsilon}:=\varepsilon^{1/(1-\gamma)} X$, which indeed satisfies the equation
\begin{equation} \label{scaledCIRCEV}
dX^{\varepsilon}_{t} =
(\alpha^{\eps} + \beta X^{\eps}_t)dt
+ \varepsilon \sigma (X^{\varepsilon}_{t})^{\gamma} dB_{t},
\qquad
X^{\varepsilon}_{0}=x^{\varepsilon}
\end{equation}  
with
\[
\alpha^{\varepsilon} :=  \varepsilon^{1/(1-\gamma)} \alpha
\qquad
x^{\varepsilon} := \varepsilon^{1/(1-\gamma)}x.
\]
Of course, this change of variables allows to write $W(X_t>R)=W(X^{\eps}_t>1)$ using $\eps=R^{-1/(1-\gamma)}$.
As mentioned above, the question is whether a large deviation principle holds \emph{at all} for $W(X^{\eps}_t \in \cdot)$ as $\eps \to 0$.
Note that both the initial condition $x^{\eps}_0$ and the constant term $\alpha^{\eps}$ in the drift coefficient tend to zero as $\eps\to 0$.
On the one hand, it is not difficult to see that $X^{\eps} \to 0$ in law with respect to the uniform topology on $C([0,T])$. 
On the other hand, writing down formally the limiting ODE that should govern the large deviations, one gets
\begin{equation} \label{degODE}
\dot{\varphi}_{t} = \beta \varphi_{t} + \sigma |\varphi_t|^{\gamma} \dot{h}_{t}, \qquad \varphi_{0} = 0.
\end{equation}
The equation \eqref{degODE} is known to admit infinitely many solutions.
When $\dot{h}_t \ge 0$, the set of solutions contains the one-parameter family $\varphi^{(\theta)}_t = e^{\beta t} \left( \sigma (1-\gamma) \int_{\theta}^t e^{-\beta(1-\gamma)s}\dot{h}_s ds \right)^{1/(1-\gamma)} 1_{\{t \ge \theta\}}$, with $\theta \ge 0$.\footnote{When $\beta=0$, $\gamma=1/2$ and $\dot{h}\equiv1$, one retrieves the textbook example of ODE for which uniqueness fails, $\dot{\varphi}_{t} = \sigma \sqrt{|\varphi_t|}$, whose solutions from $\varphi_0=0$ are given by the one-parameter family $\varphi^{(\theta)}_t = \frac{\sigma^2}4 (t-\theta)^2 1_{\{t \ge \theta \}}$.}
Then, the definition itself of the map $h \mapsto \varphi(h)$ associating the control with the corresponding solution of the ODE is not anymore possible.

We will occasionally address this situation as ``degenerate''.
Let us note straight away that large deviations for diffusions with non-Lipschitz coefficients have been studied in Baldi and Caramellino \cite{BALDI} Donati-Martin et al. \cite{DONATIYOR}, Klebaner and Lipster \cite{Klebaner} and Robertson \cite{Robertson}.
In \cite[Theorem 1.2]{BALDI} a large deviation principle is derived for the family of equations
$dX^{\eps}_t=b(X^{\eps}_t)dt+\eps \sigma(X^{\eps}_t)dB_t, \
X^{\eps}_0=x>0$
(note the strictly positive initial condition), where the function $\sigma(\cdot)$ roughly behaves like $\sigma x^{\gamma}$ (see \cite[Assumption (A1.1)]{BALDI} for precise conditions) and $b:[0,\infty)\mapsto \R$ is a locally Lipschitz function with sub-linear growth and $b(0)>0$.
The conditions for both a drift term $b$ and an initial datum independent of $\eps$, such that $b(0)>0$ and $x>0$, are violated in the situation we consider here.
In \cite{DONATIYOR}, $b(0)=0$ and $x=0$ are allowed, but the analysis is limited to the square-root case $\gamma=1/2$, and $b$ and $x$ remain independent of $\eps$.
Note in this respect that setting $b(0)=x=0$ implies $X^{\eps} \equiv 0$ for \emph{all} $\eps$, and in this case a LDP trivially holds with the rate function $I(0)=0$, $I(\phi) =\infty$ for $\phi \not\equiv 0$ (as stated in \cite[Thm 1.3]{DONATIYOR}); in contrast with \eqref{scaledCIRCEV}, where both $b^{\eps}(0)=\alpha^{\eps}$ and $x^{\eps}$ do tend to zero as $\eps \to 0$, but coming from strictly positive values, so that the solution of the SDE is non trivial for every value of $\eps$.
In both these works, uniqueness for the limiting ODE is a key point (and appears as a part of \cite[Assumption (A2.3)]{BALDI} and is exploited in \cite[Section 5]{DONATIYOR}).
In order to study the asymptotic behavior of the ruin probability $W(\tau_0 \le T)$ with $\tau_0 = \inf\{t: X_t = 0\}$ as the initial condition $x$ tends to infinity, Klebaner and Lipster \cite{Klebaner} exploit a similar space scaling by working with the `normed' process $X^{x}_t = X_t/x$, and show that a LDP holds for the process $X^{x}$ as $x \to \infty$.
The major difference with our setting is that the initial condition $X^{x}_0=1$ in \cite{Klebaner} is fixed and does not tend to zero as $x^{\eps}$ in \eqref{scaledCIRCEV}, which is one of the difficulties to encompass in our analysis.
Robertson \cite{Robertson} derives LDP for a class of stochastic volatility models, including the Heston model with square-root volatility process. 
One of the assumptions used there is that the small noise problem for the volatility process has the same form as in Donati-Martin et al. \cite{DONATIYOR}, see \cite[Assumption 2.1]{Robertson}, and the work carried out is to transfer the LDP to the second component of the process (the log-price). Therefore, the work of \cite{Robertson} does not cover small-noise problems in the form of \eqref{scaledCIRCEV}.

We establish a LDP for a generalized version of equation \eqref{scaledCIRCEV}, allowing $\alpha$ to be a function of the process.
That is, we start from equation \eqref{diffusion} under the assumptions:
\begin{itemize}
\item[(H1)]  $\gamma \in  [1/2 , 1), \sigma>0, x>0 $.
\item[(H2)] $b(y) = \alpha(y)+ \beta y$, where $ \alpha $ is a Lipschitz continuous and bounded function, and $\alpha(y) \geq 0$ in a neighbourhood of $0$.
\end{itemize}
Under (H1)-(H2), \eqref{diffusion} is known to admit a positive solution, which is pathwise unique by Yamada and Watanabe's uniqueness theorem.

\begin{theorem} \label{t:mainIntro}
Assume conditions (H1)-(H2), and let $(X_t)_{t \ge 0}$ be the unique strong solution to \eqref{diffusion}.
Set $X^{\varepsilon}:=\varepsilon^{1/(1-\gamma)} X $; then $X^{\eps}$ satisfies \eqref{scaledCIRCEV} with the constant $\alpha$ replaced by the function $\alpha(\cdot)$.
Then, the family $\{X^{\eps}\}_{\eps}$ satisfies a large deviation principle on the path space $C([0,T],\R_+)$ with inverse speed $\eps^2$ and rate function
\[
I_{T} \left( \varphi \right)=\frac{1}{2\sigma^{2}}
\int_0^T
\left( \frac{\dot{\varphi}_{t} - \beta \varphi_{t}}{ \varphi^{\gamma}_{t} } \right)^{2}
1_{\left\{\varphi_{t} \neq 0 \right\}}
dt,
\]
and $I_{T} \left( \varphi \right) = + \infty $ whenever $\varphi(0)\neq 0$ or $\varphi$ is not absolutely continuous.
\end{theorem}

\noindent
Let us note that in the definition of $I_T$ above, the expression $\frac1{{\varphi_t}^{\gamma}} 1_{\varphi_t \neq 0}$ is intended to be well defined for any $\varphi_t \in \R_{+}$, and it is equal to zero when $\varphi_t=0$.
It is easy to see that the unique zero of $I_T$ is $\varphi \equiv 0$, consistently with the fact that $X^{\eps} \stackrel{W}{\rightarrow}0$ as $\eps \to 0$.
Roughly speaking, Theorem \ref{t:mainIntro} allows to write $W(X^{\eps} \in \Gamma) = \exp \left( -\frac1{\eps^2} \Bigl( \inf_{\phi \in \Gamma} I_T(\phi)+ \psi(\eps) \Bigr) \right)$ for subsets $\Gamma$ of $C(0,T)$ such that $\inf_{\phi \in \overline{\Gamma}}I_T(\phi)=\inf_{\phi \in \accentset{\circ}{\Gamma}}I_T(\phi)$, where the function $\psi(\eps)$ vanishes as $\eps \to 0$; we refer to Theorem \ref{main} in Section \ref{s:main} for the precise statements.

According to our definition of $X^{\eps}$, one has $W(X^{\eps}_t \ge 0, \forall t \ge 0, \forall \eps >0)=W(X_t \ge 0, \forall t \ge 0)=1$.
A criterium for the strict positivity of the trajectories of $X^{\eps}$, based on Feller's test for explosion, can also be given (see \cite[Prop 3.1]{DEMARCO}: when $\gamma>1/2$, $a(0)>0$ implies $W(X^{\eps}_t>0, t \ge 0)=1$, while for $\gamma=1/2$, the same conclusion is guaranteed by $2\alpha(y)/\sigma^2 \ge 1$ for $y$ in a right neighborhood of zero - yielding the familiar Feller condition $2\alpha/\sigma^2 \ge 1$ when $\alpha$ is constant).
Note that Theorem \ref{t:mainIntro} does not assume any of these condition for the non-attainability of zero; in particular for the CIR diffusion, we do not assume the Feller condition on the coefficients $\alpha$ and $\sigma$.

From Theorem \ref{t:mainIntro}, tail asymptotics for some functionals of the process $X$ can be derived (which is exactly why the $\eps$-scaling leading to $X^{\eps}$ was introduced!).
The pathwise LDP allows to consider path functionals of the process, such as the running supremum, or the time average.

\begin{theorem} \label{t:tailIntro}
Let $(X_t)_{t \ge 0}$ be the unique strong solution to \eqref{diffusion} under conditions (H1)-(H2), and let $T>0$.
Then, as $R\rightarrow \infty$
\be \label{e:tailAs}
W \left( X_{T} \geq R \right) = 
e^{-R^{2(1-\gamma)}(c_T + o(1))}
\ee
and
\be \label{e:tailAsSup}
W \biggl( \sup_{t \in [0,T]} X_t \geq R \biggr) = 
e^{-R^{2(1-\gamma)}(c_T + o(1))}
\ee
and
\be \label{e:tailAsAsian}
W\left(\frac{1}{T}\int_{0}^{T} X_{t} dt \geq R \right) = e^{-R^{2(1-\gamma)}(\nu_T+o(1))}.
\ee
The constant $c_T$, resp. $\nu_T$ are explicitly known in terms of the model parameters, and are provided below in Proposition \ref{CIRCEVtails}, resp. Proposition \ref{LaplaceAsian} for the case $\gamma=1/2$.
\end{theorem}

\noindent
The estimates in Theorem \ref{t:tailIntro} can be compared with the explicit formulae available for cumulative distributions and critical exponents in the CIR and CEV models: these consistency checks are done in Sections \ref{s:tailAs} and \ref{s:tailAsymptProofs}, showing that the estimates in Theorem \ref{t:tailIntro} are correct on the log-scale.
While in the one-dimensional setting the large deviation approach yield by Theorem \ref{t:mainIntro} applies to equations with a more general drift term than a purely affine function, it also opens the way to heat kernel analysis for higher-dimensional diffusions involving \eqref{diffusion} as a component, which is exactly the case left open in \cite{DFJV:II}.

Let us finally note that, due to the non uniqueness of solutions for the limiting system, the problem we consider here appears to be related to the issue of regularization by noise of ODEs.
Leaving further discussions to future work, let us just point out here a structural difference with that setting: in that context, one considers an SDE of the form $dX^{\eps}_t=b(X^{\eps}_t)dt + \eps dB_t$, with unit dispersion coefficient, seen as a perturbation of the deterministic system $\dot{x}_t=b(x_t)$ with non-Lipschitz drift $b$ (e.g. $b(x)=sign(x)|x|^{\gamma}$).
Among the possible solutions of the deterministic system, one then looks at the (few) ones supporting the limiting law of $X^{\eps}$, obtaining the so-called zero noise limits of the equation; see \cite{Trevisan} and references therein.
In our framework, the equation for $X^{\eps}$ already possesses a Lipschitz continuous drift $b(x) =\alpha_{\eps} + \beta x$.
Correspondingly, the limiting system $\dot{x}_t = \beta x$, $x_0=0$, \emph{already has} a unique solution (here: the null path $x=0$), which then gives the unique weak limit for $X^{\eps}$ (in contrast to \cite[Corollary 1.2]{Trevisan}, where the limit is a probability distribution supported on two trajectories).
As we pointed out, the difficulties in our setting come from the non-Lipschitz diffusion coefficient and appear at the level of the definition of the rate function via the control system \eqref{degODE}.

In the remainder of the document, Section \ref{s:main} is devoted to the proof of Theorem \ref{t:mainIntro}, while in Section \ref{s:tailAsymptProofs} we prove the different statements of Theorem \ref{t:tailIntro}.
We collect in Appendix A the proofs of some of the more technical material.
\medskip

\emph{Acknowledgements}.
We would like to thank an anonymous referee for the careful reading of the paper and for several valuable comments which helped to improve the presentation.
We thank Peter Friz for stimulating discussions and Antoine Jacquier for useful references on integrated CIR processes.
SDM (affiliated with TU-Berlin when this work was started) acknowledges partial financial support from Matheon.
GC acknowledges financial support from Berlin Mathematical School.
SDM and GC acknowledge financial support for travel expenses from the research program `Chaire Risques Financiers' of the Fondation du Risque.

\section{Main theoretical estimates} \label{s:main}

Let $\Omega :=C \left( [0,T], \mathbb{R} \right)$, $\Omega_{\geq 0}:=C \left( [0,T], \mathbb{R}_+ \right)$  denote the space of continuous (resp. continuous non negative) functions on $[0,T]$.
$\left( \Omega, \mathcal{F}_{t}, \mathcal{F} \right)$ denotes the canonical Wiener space, $W$ the Wiener measure on $\left( \Omega, \mathcal{F}_{t}, \mathcal{F} \right)$, and $\ex $ the expectation under $W$.
We denote $H = \{ h \in AC([0,T],R): \dot{h} \in L^{2} \}$ the space of absolutely continuous paths on $[0,T]$ with square-integrable derivative (usually referred to as Cameron-Martin space).
For a set of coefficients $\alpha(\cdot),\beta,\gamma,\sigma$ satisfying conditions (H1)-(H2), we denote $X$ the $W$ almost-surely unique strong solution of \eqref{diffusion}.
We define the rescaled process $X^{\varepsilon} : = \varepsilon^{\frac{1}{1-\gamma}} X$; it is clear that $X^{\varepsilon}$ solves equation \eqref{scaledCIRCEV} with  coefficients identified by $\alpha^{\eps}(x)=\eps^{1/(1-\gamma)} \alpha(x)$ and $x^{\eps}=\eps^{1/(1-\gamma)} x$.
Denote $b^{\eps}(x) := \alpha^{\eps}(x) + x$.

The following theorem gives the precise LDP announced in Theorem \ref{t:mainIntro} in the Introduction.
We recall that the expression $\frac1{y^{\gamma}} 1_{y \neq 0}$ is well defined for any $y \in \R_{+}$, and it is equal to zero when $y=0$.

\begin{theorem} \label{main}
 Let $X^{\varepsilon}$ be the unique strong solution to \eqref{scaledCIRCEV}.
Then,
\be \label{e:LD}
\begin{aligned}
\limsup_{\varepsilon \rightarrow 0}  \varepsilon^{2} \log W(X^{\varepsilon} \in F) &\leq -\inf_{F} I_{T} \left( \varphi \right)
\\
\liminf_{\varepsilon \rightarrow 0}  \varepsilon^{2} \log  W(X^{\varepsilon} \in G)
&\geq - \inf_{G} I_{T} \left( \varphi \right)
\end{aligned}
\ee
for every closed set $F \subseteq \Omega_{\geq 0} $ and every open set $G \subseteq \Omega_{\geq 0} $, where the rate function $I_{T} \left( \varphi \right) $ is defined by
\begin{equation} \label{e:rateFct}
I_{T} \left( \varphi \right):=\frac{1}{2\sigma^{2}}
\int_0^T
\left( \frac{\dot{\varphi}_{t} - \beta \varphi_{t}}{ \varphi^{\gamma}_{t} } \right)^{2}
1_{\left\{\varphi_{t} \neq 0 \right\}}
dt,
\end{equation} 
and $I_{T} \left( \varphi \right) = + \infty $ whenever $\varphi(0)\neq 0$ or $\varphi$ is not absolutely continuous.
\end{theorem}

\begin{remark}
We could state the large deviation principle of Theorem \ref{main} on $\Omega=C([0,T],\R)$, setting the rate function $I_T(\varphi)$ to $+\infty$ whenever $\varphi \notin \Omega_{\ge 0}$.
Since the process $X^{\eps}$ is known to be positive $W$-$a.s.$ for every $\eps>0$, with such a definition of the rate function the LDP \eqref{e:LD} holds for every closed subset $F$ and every open subset $G$ of $\Omega$.
\end{remark}

\begin{remark}
As pointed out in the Introduction, the rate function for a family $\{X^{\eps}\}_{\eps}$ satisfying $dX^{\eps} = b(X^{\eps})dt + \eps \sigma(X^{\eps})dB_{t}, \ X^{\eps}_{0}=x$, can be written as
\be \label{e:rateFctDeg}
\overline{I}_T(\varphi) = \inf \Bigl\{\frac12 |\dot{h}|_{L^2}: h \in H, \varphi(h) = \varphi \Bigr\}
\ee
where $\varphi(h)$ is the solution to the limiting ODE controlled by $h$, $\dot{\varphi}=b(\varphi)+\sigma(\varphi)\dot{h}$ and $\varphi_0=x$, provided this solution is unique.
In our setting, consider $\varphi \in S(u)$, where now $S(u)$ denotes the set of positive solutions of the degenerate ODE \eqref{degODE} with control parameter $h=u \in H$: on the set $\{\varphi>0\}$, $u$ is uniquely determined by $\varphi$ via $\dot{u}_t=\frac{\dot{\varphi}_t-\beta \varphi_t}{\varphi_t^{\gamma}}$; on the set $\{\varphi=0\}$, the function $\varphi$ is seen to satisfy equation \eqref{degODE} for \emph{any} control parameter $h$.
This means that the set of $h$ such that $\varphi \in S(h)$ contains the infinitely many elements given by
\[
\dot{h}_t = \frac{\dot{\varphi}_t-\beta \varphi_t}{\varphi_t^{\gamma}} 1_{\{\varphi_t>0\}}
+ \frac{d\tilde{h}_t}{dt} 1_{\{\varphi_t=0\}},
\qquad \tilde{h} \in H.
\]
The control $h_0$ achieving the minimum norm is obtained setting $\tilde{h}\equiv 0$.
This gives $\frac12 |\dot{h}_0|_{L^2}=
\inf \bigl\{\frac12 |\dot{h}|_{L^2}: h \in H,  \varphi \in S(h) \bigr\}=I_T(\varphi)$ for the rate function $I_T$ defined in \eqref{e:rateFct}.
\end{remark}

\begin{remark}
Assume that $b:[0,\infty) \to \R$ is a locally Lipschitz function with sublinear growth and $b(0)>0$, and that $\overline{X}^{\eps}$ satisfies $d\overline{X}^{\eps}_t=b(\overline{X}^{\eps}_t)dt+ \eps \sigma \bigl(\overline{X}^{\eps}_t\bigr)^{\gamma}dB_t$ and $\overline{X}^{\eps}_0=x>0$.
Then it is known from \cite[Thm 2.1]{BALDI} or \cite[Thm 4.2]{ChiariniFischer} that $\overline{X}^{\eps}$ satisfies a LDP with rate function
\[
J_T(\varphi):=\frac{1}{2\sigma^{2}}
\int_0^T
\left( \frac{\dot{\varphi}_{t} - b(\varphi_{t})}{ \varphi^{\gamma}_{t} } \right)^{2}dt,
\]
and $J_T(\varphi)=\infty$ if $\varphi$ is not absolutely continuous, where one classically agrees that $1/\varphi_t$ is equal to $+\infty$ if $\varphi_t=0$.
We stress that the latter rate function is radically different from $I_T$ defined in \eqref{e:rateFct}: whenever $\varphi=0$ on some non trivial interval $K \subset [0,1]$, then $J_T(\varphi)=\infty$, while in such a case the integrand in \eqref{e:rateFct} gives zero contribution to $I_T$ on $K$.
In other words, while trajectories with a zero-set of positive measure require infinite energy to be followed by the process $\overline{X}^{\eps}$ in the small-noise limit, they are favoured by the rate function of the process $X^{\eps}$.
\end{remark}

\subsection{Tail asymptotics} \label{s:tailAs}
 
The space-scaling $X^{\eps}=\eps^{1/(1-\gamma)}X$ together with the large deviation principle \eqref{e:LD} allow to work out tail asymptotics for functionals of the process $X$.
The following proposition provides the precise constants appearing in Theorem \ref{t:tailIntro} in the Introduction.

\begin{proposition} \label{CIRCEVtails}
The asymptotic formulas \eqref{e:tailAs} and \eqref{e:tailAsSup} in Theorem \ref{t:tailIntro} hold with the constant $c_T$ given by
\begin{equation} \label{e:leadingOrder}
c_T = \begin{cases}
\frac{\beta e^{-2\beta(1-\gamma)T}}{\sigma^{2}(1-\gamma) (1-e^{-2\beta(1-\gamma)T})} & \text{if $\beta \neq 0$}
\\
\frac{1}{2\sigma^{2}(1-\gamma)^{2}T} & \text{if $\beta =0$.}
\end{cases}
\end{equation}
\end{proposition}
One can see that $c_T$ does not depend on the function $\alpha(\cdot)$ in the drift of $X$, nor on the initial condition $x$.

\begin{remark}
Some comments are in order.

\begin{itemize}
\item[(i)] \emph{\textbf{Comparison with explicit formulae for the CEV process}.
The asymptotic behavior \eqref{e:tailAs} can be compared with the explicit formulae available for the density of the CEV process.
When $\alpha\equiv0$ in \eqref{diffusion}, $X$ can be obtained as a deterministic time-change of a power of a squared Bessel process (see \cite[Section 6.4.3]{JEANBLANC}).
As a consequence, 
for every $T>0$ the random variable $X_T$ is known to admit a density with respect to the Lebesgue measure on the positive real line, given by
\be \label{e:densityCEV}
\begin{aligned}
f_{X_T}(y) &= 
\frac{(1-\gamma)}{d(T)}
e^{\beta(-2(1-\gamma)+1/2)T}
\exp\left(-\frac{1}{2d(T)} \bigl( x^{2(1-\gamma)}+y^{2(1-\gamma)}e^{-2\beta(1-\gamma)T}\right)
\\
&\quad\quad \times
x^{1/2} y^{-2\gamma+1/2}
I_{1/2(1-\gamma)}
\left(\frac{1}{d(T)} x^{1-\gamma} y^{1-\gamma}e^{-\beta(1-\gamma)T}\right),
\qquad y >0,
\end{aligned}
\ee
where $I_{\nu}$ is the modified Bessel function of the first kind of index $\nu>0$, and $d(T)=\frac{(1-\gamma)\sigma^2}{2\beta}(1-e^{-2\beta(1-\gamma)T})$ (note \emph{en passant} that one has $d(T)>0$ for every choice of the sign of $\beta$).\footnote{When $\gamma \in [1/2,1)$, the law of $X_T$ also possesses an atom at zero, $\Prob(X_T=0)=m_T>0$, and an explicit formula for the mass $m_T$ is available (see again \cite[Chap.6]{JEANBLANC}).
From our point of view, this only means that the density $f_{X_T}$ does not integrate to $1$ on $(0,\infty)$, without affecting our analysis of the tail asymptotics at $\infty$.}
The formula \eqref{e:densityCEV} is also valid for $\beta=0$, when one replaces all the $\beta$-dependent constants with their limits as $\beta\to 0$, such as $d(T)|_{\beta=0}=(1-\gamma)^2\sigma^2T$.
Using the asymptotic behavior (see~\cite[Section 9.7.1]{ABRAMOWITZ}) of the modified Bessel function $I_{\nu}(z) \sim \frac{e^z}{\sqrt{2 \pi z}}$ as $z\to \infty$ for fixed $\nu>0$, one immediately obtains
\[
\log f_{X_T}(y) =: g(y) \sim -\frac{e^{-2\beta(1-\gamma)T}}{2d(T)}y^{2(1-\gamma)} = -c_T y^{2(1-\gamma)},
\qquad x \to \infty,
\]
with the constant $c_T$ defined in \eqref{e:leadingOrder}.
Using some standard tools of regular variation \cite{Bingham}, one can then easily prove that $\log W(X_T>y)=\log\int_y^{\infty} e^{g(z)} dz \sim g(y) \sim -c_T y^{2(1-\gamma)}$ as $y \to \infty$, thus showing that estimate \eqref{e:tailAs} is exact on the log-scale.
}

\item[(ii)] \emph{The asymptotic estimate $f_{X_T}(y) \le A_T \: e^{-a_T y^{2(1-\gamma)}}$, $y > 1$, for the density of $X_T$ was proven in \cite{DEMARCO} for the solutions of a class of SDEs containing \eqref{diffusion} under conditions (H1)-(H2) (namely, in \cite{DEMARCO} the coefficients $\beta$ and $\gamma$ are also allowed to depend smoothly on $X$), relying on techniques of Malliavin calculus and transformations for 1-dimensional SDEs.
The constant $a_T$ provided there is not optimal.
While the estimates in \cite{DEMARCO} remain valid for more general equations, the large deviation principle in Theorem \ref{main} allows to obtain a sharp estimate on the log-scale.}
\end{itemize}
\end{remark}
\medskip

The asymptotic behavior $W\bigl(\frac1T \int_0^T X_t dt\bigr)=\exp\bigl(-R^{2(1-\gamma)}(\nu_T+o(1))\bigr)$ for the time average of the process can also be proven using Theorem \ref{main}: see Proposition \ref{LaplaceAsian} in Section \ref{s:tailAsymptProofs}, where an expression of the constant $\nu_T$ is provided in the case $\gamma=1/2$.

\section{Proof of the main estimates}

We prove the large deviation principle in Theorem \ref{main} by first showing the exponential tightness of the family $\{X^{\eps}\}_{\eps}$, namely for every $m<0$ there exists a compact set $K_m \subset C([0,T])$ such that $\limsup_{\eps \to 0} \eps^2\log W(X^{\eps} \in K_m^c) \le m$.
We then prove the weak upper bound
\[
\limsup_{R \to 0} \: \limsup_{\eps \to 0} \eps^2 \log W(X^{\eps}\in B(\varphi,R)) \le -I_T(\varphi)
\qquad \forall \varphi \in \Omega_{\ge 0},
\]
and the weak lower bound
\[
\liminf_{R \to 0} \liminf_{\eps \to 0} \eps^2 \log W(X^{\eps}\in B(\varphi,R)) \ge -I_T(\varphi)
\qquad \forall \varphi \in \Omega_{\ge 0}
\]
where $B(\varphi,R)$ denotes the closed ball in $C([0,T])$ of radius $R$, $B(\varphi,R):=\{\tilde{\varphi}: |\tilde{\varphi}-\varphi|_{\infty} \le R\}$.
It is a general fact that exponential tightness combined with the weak upper bound yields the large deviation upper bound in \eqref{e:LD} for any closed set after a covering argument (see \cite[Chapters 1 and 2]{DEUSCHEL}).
On the other hand, the weak lower bound trivially provides the full lower bound in \eqref{e:LD}, observing that open sets are neighborhoods of their points.

\subsection{Exponential tightness} \label{s:expTight}

We prove the exponential tightness considering balls in the \Holder norm $\| \omega \|_{\eta} := \sup_{s,t \leq T, s\neq t} \frac{|\omega_{t}- \omega_{s} |}{|t-s|^{\eta}}$ and a natural bound on the initial condition $\omega_0$.
More precisely, we define 
\be  
K_R:= \{ \| \omega \|_{\eta} \leq R \} \cap \{ \omega_0 \in (0, x] \}.
\ee
It is classical that these sets are compact in $C([0,T])$.

\begin{proposition} \label{exptight}
The family of measures 
$W (X^{\eps} \in \cdot)$ is exponentially tight in scale $\varepsilon^{2}$, i.e.
\begin{equation*}
\lim_{R \rightarrow + \infty }\limsup_{\varepsilon \rightarrow 0 }\varepsilon^{2} \log W \left(X^{\eps} \in K^c_R \right) = -\infty
\end{equation*}
for every $0<\eta<\frac{1}{2}$.
\end{proposition}

\noindent
We follow \cite{DONATIYOR} in the proof of Proposition \ref{exptight}.
First, let us observe that for $\varepsilon \leq 1$, $W(X^{\eps}_0 \in (0,x]) = 1$ so that we just need to estimate the \Holder norm of $X^{\eps}$.
To this end, we use a version of Garsia-Rodemich-Rumsey's Lemma, and the existence of exponential moments for a process bounding $X^{\eps}$ from above.

\begin{lemma} \label{expmoments}
Consider $(\tilde{X_t},t \ge 0)$ the strong solution to
\[
d\tilde{X}_t = (|\alpha|_{\infty}+|\beta|\tilde{X}_t)dt + \sigma(\tilde{X}_t)^{\gamma} dB_t, \qquad \tilde{X}_0 = x
\]
and define $\tilde{X}^{\eps}:=\eps^{1/(1-\gamma)} \tilde{X}$.
Then, there exist positive constants $c$ and $C$ such that:
\be \label{e:expmoments}
\ex \left(\exp{ \left(c \varepsilon^{-2} (\tilde{X}^{\varepsilon}_{t})^{2(1-\gamma)} \right)} \right) \leq C, \quad \forall t \in [0,T], \quad \forall \eps>0.
\ee
\end{lemma}
\begin{proof}
According to the definition of $\tilde{X}^{\eps}$, one has $\varepsilon^{-2} (\tilde{X}^{\varepsilon}_{t})^{2(1-\gamma)}=\tilde{X}^{2(1-\gamma)}_{t}$, so that \eqref{e:expmoments} holds if and only if $\ex \left[ \exp{\left(c \tilde{X}_{t}^{2(1-\gamma)} \right)} \right] \leq C$ for all $t \in [0,T]$.
When $\gamma=1/2$, \eqref{e:expmoments} follows from the asymptotic behavior of the density of the CIR process for large arguments (see e.g. \cite[section 6.3.2 p.358]{JEANBLANC}); for general $\gamma$ and $\beta=0$, from the asymptotic behavior of the density of the classical CEV process as stated for example in \cite[Lemma 6.4.3.1 p.368]{JEANBLANC}.
For general $\gamma$ and $\beta$, we rely on a slight generalization of the proof of \cite[Prop 3.3]{DEMARCO}; we leave the details to Appendix \ref{a:appendix1}.
\end{proof}
\medskip

The next proposition is a direct consequence of Garsia-Rodemich-Rumsey's Lemma; see  Appendix \ref{a:appendix1} for a statement of this lemma and a proof of Proposition \ref{GRRI}.

\begin{proposition} \label{GRRI}
Let $\omega \in\Omega $.
Fix $\varepsilon,R>0$, $\eta \in (0, \frac{1}{2})$. 
Assume that: 
\be \label{e:Garsia}
\int_{0}^{T} \int_{0}^{T} \exp{\left(
\frac{ | \omega_t-\omega_s | }{\varepsilon^{2}\sqrt{|t-s|}}\right)}dsdt   \leq  K_{\varepsilon,\eta}(R)
\ee
with $K_{\varepsilon,\eta}(R):=\frac{1}{4} \exp \left(  T^{\eta-1/2}\left(\frac{R}{8\varepsilon^{2}} -4T^{1/2-\eta} - K_{\eta}\right) \right) -\frac{1}{4}T^{2}$ and $K_{\eta}:= \sup_{u \in [0,T]} 2{u}^{1/2-\eta}\log(u^{-1})<\infty$.
Then,
\begin{equation} \label{e:GarsiaHoldNorm}
\| \omega \|_{\eta} \leq R.
\end{equation}
\end{proposition}

In the proof of Proposition \ref{exptight}, we exploit a localization procedure: for any $\varepsilon>0$ and $n \in \mathbb{N}$, define the process $X^{\varepsilon,n}$ as the strong solution of the SDE with truncated coefficients:
\begin{equation}\label{eq:crsp}
dX_{t}^{\varepsilon,n}= b^{\varepsilon}(X^{\varepsilon,n}_{t} \wedge n) dt + \sigma \varepsilon \left({X^{\varepsilon,n}_{t}} \wedge n \right)^{\gamma} dB_t , \quad
X^{\varepsilon,n}_{0}=x^{\varepsilon}.
\end{equation}
The paths of $X^{\varepsilon,n}$ can be decomposed in their martingale part and locally bounded variation part
\begin{equation*}
d X^{\varepsilon,n}_{t} = dA^{\varepsilon,n}_t+dM^{\varepsilon,n}_t
\end{equation*}
with $dM^{\varepsilon,n}_t = \varepsilon \sigma (X^{\varepsilon,n}_t \wedge n)^{\gamma} dB_t$ and $dA^{\varepsilon,n}_t = b^{\varepsilon}(X^{\varepsilon,n}_{t} \wedge n) dt$.
We shall also define for every $n, \varepsilon$ the stopping time $T^{\varepsilon, n}  := \inf \left\{ t \geq 0 :  X^{\varepsilon}_t \geq n\right\}$.
By the pathwise uniqueness for equation \eqref{diffusion} (equivalently, \eqref{eq:crsp}), we have that up to time $T^{\varepsilon, n}$ the processes $\left( X^{\varepsilon}_t \right)_{t \in [0,T]}$ and $\left(X^{\varepsilon,n}_t \right)_{t \in [0,T]}$ coincide almost surely.
More precisely, $\forall n \in \mathbb{N}$ and $\varepsilon>0$
\be \label{e:pu}
W \left( X^{\varepsilon}_{t \wedge T^{\varepsilon,n}} = X^{\varepsilon,n}_{t \wedge T^{\varepsilon,n}}, \forall t \in [0,T] \right) = 1.
\ee

\begin{proof}[Proof of Proposition \ref{exptight}]
Let us fix $\eta \in (0,\frac{1}{2})$.
By $\eqref{e:pu}$,
\begin{align} \label{localization}
 \nonumber  W\left( \|X^{\eps} \|_{\eta} \geq R\right) & \leq W\left( \| X^{\varepsilon,n}\|_{\eta} \geq R, T^{\varepsilon,n} \geq T \right)  +  W \left( T^{\varepsilon,n} \leq T\right) \\
  & \leq  W \left( \| X^{\varepsilon,n} \|_{\eta} \geq R  \right) + W\left( T^{\varepsilon,n} \leq T\right).
 \end{align}
Let us estimate the first term in \eqref{localization}.
Using Proposition \ref{GRRI} and Markov's inequality we have for every $\varepsilon, n$:
\begin{align*}
W (\| M^{\varepsilon,n} \|_{\eta} \geq R) & \leq W \left( \int_{0}^{T} \int_{0}^{T} \exp\left(\varepsilon^{-2}\frac{ | M^{\varepsilon,n}_t-M^{\varepsilon,n}_s | }{\sqrt{|t-s|}} \right) dsdt
\geq
K_{\varepsilon,\eta}(R) \right) \\ 
&\leq \frac{1}{K_{\varepsilon,\eta}(R)} \int_{0}^{T} \int_{0}^{T} \ex \left( \exp\left(\varepsilon^{-2}\frac{ | M^{\varepsilon,n}_{t}-M^{\varepsilon,n}_{s} | }{\sqrt{|t-s|}} \right)\right)dsdt.
\end{align*}
Applying the exponential martingale inequality $\esp \left( \exp( \lambda M_t) \right) \leq \sqrt{\esp \left(  \exp \left( 2 \lambda^{2} \langle M \rangle_t \right) \right)}$ \cite[Chap IV]{REVUZYOR} with $\lambda=\frac 1{\eps^2\sqrt{|t-s|}}$, for $t>s$ one has 
\[
\begin{aligned}
\ex \left( \exp\left( \frac{|M^{\varepsilon,n}_t-M^{\varepsilon,n}_s |}{\eps^2 \sqrt{t-s}} \right)\right)
& \leq 
2 \sqrt{\ex \left( \exp\left(\frac{2\sigma^2}{\varepsilon^2(t-s)}\int_{s}^{t}\left( X^{\varepsilon,n}_{r} \wedge n \right)^{2\gamma} dr\right) \right) }  \\
& \leq 2 \exp \left(  \sigma^{2}\varepsilon^{-2}n^{2\gamma} \right).
\end{aligned}
\]
Therefore, using the definition of the constant $K_{\eps,\eta}(R)$ in Proposition \ref{GRRI}
\begin{align}\label{Fernique}
\limsup_{\varepsilon \rightarrow 0} \varepsilon^{2} \log W   \left( \| M^{\varepsilon,n} \|_{\eta} \geq R  \right) \leq -T^{\eta - 1/2 } \frac{R}{8}+\sigma^2 n^{2\gamma}.
\end{align}
For the bounded variation part $A^{\varepsilon,n}$, we observe that
\begin{align*}\label{BV1}
W \left( \| A^{\varepsilon,n} \|_{\eta} \geq R \right) \leq W \left(T^{1-\eta} \sup_{t \in [0,T]} b^{\varepsilon} \left( X^{\varepsilon,n}_{t} \wedge n \right) \geq R \right).
\end{align*}
Under hypothesis (H), $b^{\eps}(x) \leq |\alpha|_{\infty}+ \beta x$ for every $x$.
Therefore, for every $\varepsilon, n$
\begin{equation}\label{BV2}
W \left( \| A^{\varepsilon,n} \|_{\eta} \geq R \right) \leq W \left( T^{1-\eta}  (|\alpha|_{\infty} + \beta n) \geq R \right)
= 0,
\end{equation}
where the last identity holds as soon as $R>T^{1-\eta}  (|\alpha|_{\infty} + \beta n)$.

We now deal with the second term in \eqref{localization}. 
It follows from the comparison theorem for one-dimensional SDEs \cite[Proposition 5.2.18]{KS}, that $X^{\varepsilon}_{t} \leq \tilde{X}^{\varepsilon}_{t}, t \le T$, almost surely, where $\tilde{X}^{\varepsilon}$ is defined in Lemma \ref{expmoments}.
For every fixed $\gamma$ and $a>0$, it is a simple exercise to show that the function $y \mapsto \exp (a \varepsilon^{-2} (1+y)^{2(1-\gamma)})$, $y>0$, is increasing and convex if $\eps$ is small enough\footnote{The second derivative reads $e^{a \varepsilon^{-2} (1+y)^{2(1-\gamma)}} \times 2a\eps^{-2}(1-\gamma)(1+y)^{-2\gamma} \times [1-2\gamma+\frac{2a}{\eps^2}(1-\gamma)(1+y)^{2(1-\gamma)}]$.}.
For such values of $\eps$, since $\tilde{X}^{\varepsilon}_{t}$ is a submartingale, so is $\exp \left( a \varepsilon^{-2}(1+\tilde{X}^{\varepsilon}_{t})^{2(1-\gamma)} \right)$.
Then, we can apply Markov's inequality and Doob's $L^2$-inequality, obtaining:
\begin{align} \nonumber
W\left( T^{n,\varepsilon} \leq T\right) &= W\left( \sup_{t\in [0,T]} X^{\varepsilon}_{t}\geq n\right)
\\
& \nonumber
\leq W \left( \sup_{t\in [0,T]} \exp\left(a\varepsilon^{-2} {\left(1+ \tilde{X}^{\varepsilon}_{t}\right)}^{2(1-\gamma)}\right)\geq \exp \left( a\varepsilon^{-2}(1+n)^{2(1-\gamma)}\right)\right) 
\\
& \label{doob} \leq
\exp \left(-a \varepsilon^{-2}(1+n)^{2(1-\gamma)}\right)
\times
4 \: \ex \left(\exp \left( a \varepsilon^{-2} (1+\tilde{X}^{\varepsilon}_{T})^{2(1-\gamma)} \right)\right).
\end{align}
Using the elementary inequality $\exp(a(1+y)^{2(1-\gamma)}) \le \exp(a 2^{2(1-\gamma)})+\exp(a(2y)^{2(1-\gamma)})$,
and choosing $a$ such that $a \times 2^{2(1-\gamma)}=c$ where $c$ is the constant in Lemma \ref{expmoments}, it follows from this lemma and estimate \eqref{doob} that
\be \label{doob2} 
W\left( T^{n,\varepsilon} \leq T\right)
\le
\exp\left(-a \varepsilon^{-2}n^{2(1-\gamma)}\right)
\times 4 \left[\exp(c \eps^{-2}) + C \right],
\ee
where $C$ is the second constant in Lemma \ref{expmoments}.
Now choosing $ n:=\lfloor\sqrt{R} \rfloor $, the condition under which \eqref{BV2} holds true is satisfied for $R$ large enough. 
Passing to the limit as $\varepsilon \rightarrow 0$ in \eqref{localization} and using \eqref{Fernique}, \eqref{BV2} and \eqref{doob2}, we obtain
\begin{displaymath}
\limsup_{\varepsilon \rightarrow 0} \varepsilon^{2} \log \left(W \left( \| X^{\varepsilon} \|_{\eta} \geq R \right) \right) \leq \max \left\{ -\frac{R}{8} + \sigma^2{R}^{\gamma}, -a{{R}^{(1-\gamma)}+c} \right\}.
\end{displaymath}
Letting $R\rightarrow \infty$, the conclusion follows.
\qedBl
\end{proof}

\subsection{Weak upper bound }

This section is devoted to the proof of the following proposition.

\begin{proposition}\label{p:weakupperbound}
$\forall \varphi \in \Omega_{\ge 0} \cap H$:
\begin{equation}\label{e:weakupperbound}
\limsup_{R \rightarrow 0 } \limsup_{\varepsilon \rightarrow 0 } \varepsilon^{2} \log W \left( X^{\varepsilon} \in B (\varphi, R )  \right) \leq - I_{T} (\varphi).
\end{equation}
\end{proposition}
\bigskip

For every $h \in H,\varepsilon>0$ and $\phi \in \Omega_{\ge0}$, define
\begin{equation}\label{expmartfunctional}
F^{\varepsilon}\left( \phi , h \right) :=
h_{T} \phi_{T} - h_{0} \phi_{0}
- h_T \int_0^T b^{\eps}(\phi_s)ds
- \int_{0}^{T} \Bigl( \phi_{s} - \int_{0}^{s}b^{\varepsilon} \left( \phi_{r} \right)dr \Bigr) \dot{h}_{s}ds -\frac{\sigma^{2}}{2}\int_{0}^{T} h^{2}_{s}\phi_{s}^{2\gamma} ds.
\end{equation}
By setting $\varepsilon = 0$  in \eqref{expmartfunctional}, we can define the functional $F^{0}\left( \phi , h \right)$.
Note that $F^{\varepsilon}\left(\cdot,h\right) $ is continuous $\forall h \in H$ on the whole space $\Omega_{\ge0}$ with respect to the sup-norm topology, and converges to $F^{0} \left(\cdot,h \right)$ uniformly on $\Omega_{\ge0}$ as $\eps \to 0$.

\begin{remark} \label{r:parts}
Applying the integration by parts formula to the product $h_t X^{\eps}_t$, one has
\[
\begin{aligned}
\eps \sigma \int_0^T h_t (X^{\eps}_t)^\gamma dB_t
&= 
h_T X^{\eps}_T - h_0 x^{\eps}_0
-\int_0^{T} [\dot{h}_t X^{\eps}_t + h_t b^{\eps}(X^{\eps}_t)] dt
\\
&= 
h_T X^{\eps}_T - h_0 x^{\eps}_0
-h_T \int_0^T b^{\eps}(X^{\eps}_t) dt 
-\int_0^{T} \dot{h}_t \Bigl(X^{\eps}_t - \int_0^t b^{\eps}(X^{\eps}_s) ds\Bigr) dt,
\end{aligned}
\]
hence

\[
F^{\eps}(X^{\eps}_{\cdot},h) =
\varepsilon \sigma \int_{0}^{T} h_{s} (X^{\varepsilon}_{s})^{\gamma} dB_{s}   -\frac{\sigma^{2}}{2}\int_{0}^{T} h^{2}_{s} (X^{\varepsilon}_{s})^{2\gamma} ds.
\]
\end{remark}

\noindent
According to Remark \ref{r:parts}, the random variable
\begin{equation}\label{explocmart}
M^{\varepsilon,h}_{T}(\omega):=
 \exp \left(\frac{1}{\varepsilon^2}F^{\varepsilon}\left( X^{\varepsilon}(\omega) , h \right) \right)
\end{equation}
is the value at time $T$ of the local exponential martingale associated to $\frac{\sigma}{\varepsilon}\int_{0}^{.} h_{s} ( X^{\varepsilon}_{s})^{\gamma} dB_{s}$.
It should be stressed that, for any $h \in H$ and $\varepsilon>0$, the functionals $F^{\eps}(\phi,h)$ and $M^{\varepsilon,h}_{T} \left( \phi \right)$ are well defined for every $\phi \in \Omega_{\ge0}$, and not only almost surely. 

\begin{proof}[Proof of Proposition \ref{p:weakupperbound}]
Since any positive local martingale is a supermartingale, we have
\begin{equation} 
\ex\bigl[M^{\varepsilon,h}_{T} \big] \leq 1.
\end{equation}
Fix now a trajectory $\varphi \in \Omega_{\geq 0}$.
Using the remark above:
\begin{align*}
W(X^{\varepsilon} \in B(\varphi, R) )
&=\esp\left[ e^{-\frac{1}{\varepsilon^2}F^{\varepsilon}\left( X^{\varepsilon}, h \right)}
M^{\varepsilon,h}_{T}
1_{\{ X^{\varepsilon} \in B(\varphi, R)\}} \right]
\\
& \leq
\sup_{ \phi \in B(\varphi,R) } \exp \left(-\frac{1}{\varepsilon^2}F^{\varepsilon} \left(\phi ,h \right)\right)\ex\left( M^{\varepsilon,h}_{T}\right) 
\\ & \leq 
\sup_{\phi \in B(\varphi,R)}\exp\left( -\frac{1}{\varepsilon^{2}}  F^{\varepsilon} \left( \phi, h \right) \right). 
\end{align*}
Since $\sup_{\phi \in B(\varphi,R)}|F^{\varepsilon}\left(\phi,h\right)-F^{0}\left( \phi , h\right)| \rightarrow 0$,
we have that
\[
\limsup_{\varepsilon \rightarrow 0} \varepsilon^2 \log W (X^{\varepsilon} \in B(\varphi, R)) \leq \sup_{\phi \in B(\varphi,R)} (-F^{0}\left(\phi, h\right)).
\]
Therefore, by the continuity of $\phi \mapsto F^{0}\left( \phi,h\right)$,
\begin{align*}
\limsup_{R \rightarrow 0} \limsup_{\varepsilon \rightarrow 0} \varepsilon^2 \log( W (X^{\varepsilon} \in B(\varphi, R) \leq -F^{0}(\varphi, h), \qquad \forall h \in H.
\end{align*}
In the next proposition we prove that:
\begin{equation*}
\sup_{h \in H} F^{0}(\varphi, h) = I_{T}(\varphi)
\end{equation*}
which concludes the proof of \eqref{e:weakupperbound}.
\qedBl
\end{proof}

\begin{proposition}
$\forall \ \varphi \in  \Omega_{\geq 0 } $ we have that: 
\begin{align}\label{VariationalRate1}
\sup_{h \in H} F^{0}(\varphi, h) = I_{T}(\varphi)
\end{align}
\end{proposition}

\begin{proof}
Assume $\varphi \in \Omega_{\ge 0} \cap H$ is such that $I_T(\varphi)<\infty$.
Then, the function $u$ defined by $u_0=0, \dot{u}_s=\frac{\dot{\varphi}_s - b\varphi_s}{\sigma \varphi_s^{\gamma}} 1_{\varphi_s\neq 0}$ is by definition an element of $H$, and $\varphi$ satisfies by construction the ODE \eqref{degODE} with control $u$.
Repeating the computations in Remark \ref{r:parts}, one can see that
\[
F^{0}(\varphi,h) = \sigma \int_0^T h_s \varphi_s^{\gamma} \dot u_s ds- \frac{\sigma^2}2 \int_0^T h_s^2 \varphi_s^{2\gamma} ds.
\]
Note that $F^{0}(\varphi,h)$ is concave in $h$, hence if it has a critical point, this must be a maximum.
The Fr\'echet differential $D^h F^{0}(\varphi,h)$ at $h$, applied to the generical element $k \in H$, reads
\[
D^h F^{0}(\varphi,h)[k] = \sigma \int_0^T k_s \left[ \varphi_s^\gamma \dot u_s - \sigma h_s \varphi_s^{2\gamma} \right] ds.
\]
Therefore, $D^h F^{0}(\varphi,h)|_{h=h^*}=0$ at any $h^*$ such that $h^*_s=\frac{\dot u_s}{\sigma \varphi_s^\gamma}$ on $\{s:\varphi_s \neq 0\}$ (while $h^*_s$ can take any arbitrary value on $\{s:\varphi_s=0\}$).
For such $h^*$, one has
\[
F^{0}(\varphi,h^*) = \int_0^T (\dot u_s)^2 1_{\varphi_s\neq 0} ds - \frac12 \int_0^T (\dot u_s)^2 1_{\varphi_s\neq 0} ds
= \frac12 \int_0^T (\dot u_s)^2 1_{\varphi_s\neq 0} ds
=I_T(\varphi).
\]
On the other hand, if $\varphi$ is absolutely continuous and such that $I_T \left( \varphi \right) = + \infty$, one can approximate the function
$\frac{\dot{\varphi}_{s} - \beta \varphi_{s}}{\varphi^{2\gamma}_{s}} $ with a sequence $ h^{n}  \in H $ such that $F^{0} \left( \varphi, h^{n} \right) \rightarrow + \infty$. 
\qedBl
\end{proof}

\subsection{Weak lower bound} \label{s:wlb}

This section is devoted to the proof of 

\begin{proposition} \label{p:wlbX}
For all $\varphi \in \Omega_{\ge0}$, we have
\begin{equation}\label{weaklowerbound}
\liminf_{R \rightarrow 0 } \liminf_{\varepsilon \rightarrow 0 } \varepsilon^{2} \log W \left( X^{\varepsilon} \in B (\varphi, R )  \right) \geq - I_{T}(\varphi).
\end{equation}
\end{proposition}

\noindent
In the spirit of Lamperti's transformation, we introduce the process $Y^{\varepsilon}: = (X^{\varepsilon})^{1-\gamma}$.
$Y^{\varepsilon}$ satisfies a SDE with constant diffusion coefficient and a drift coefficient that we will be able to control.
We will prove a large deviation weak lower bound for $Y^{\varepsilon}$, and then transfer it to $X^{\eps}$ by means of the contraction principle.

\begin{proposition} \label{p:weakLowerY}
Define 
\[
\mathcal{I}_{T}(\psi) := \frac{1}{2\sigma^{2} (1-\gamma)^{2} } \int_{0}^{T} \left( \dot{\psi}_{t} - \beta(1-\gamma)  \psi_{t} \right)^{2}dt
\]
for $\psi \in \Omega_{\ge0}$, where $\mathcal{I}_{T}(\psi)=+\infty$ if $\psi(0)\neq0$ or $\psi$ is not absolutely continuous.
Then, for all $\psi $ such that $\mathcal{I}_T(\psi)<+\infty$, one has
\begin{equation}\label{weaklowerboundY}
\liminf_{R \rightarrow 0 } \liminf_{\varepsilon \rightarrow 0 } \varepsilon^{2} \log W \left( Y^{\varepsilon} \in B (\psi, R )  \right) \geq  -\mathcal{I}_{T}(\psi).
\end{equation}
In other words, the family $Y^{\varepsilon}$ satisfies a large deviation weak lower bound on $C([0,T],\R_+)$, with rate function $\mathcal{I}_T(\psi)$.
\end{proposition}

Once we are provided with Proposition \ref{p:weakLowerY}, it is straightforward to prove the weak lower bound for $X^{\eps}$.

\begin{proof}[Proof of Proposition \ref{p:wlbX}]
Consider $\psi \in \Omega_{\ge 0}$ absolutely continuous.
By Lemma 3.45 in \cite{LeoniSobolev}, $\dot{\psi}=0 \: a.s.$ on $\{\psi=0\}$. 
Therefore, $\mathcal{I}_{T}$ defined in Proposition \ref{p:weakLowerY} can be rewritten as $\mathcal{I}_{T}(\psi) = \frac{1}{2\sigma^{2} (1-\gamma)^{2} } \int_{0}^{T} \bigl( \dot{\psi}_{t} - \beta(1-\gamma)  \psi_{t} \bigr)^{2} 1_{\psi_t \neq 0} dt$.
Using the definition of $Y^{\varepsilon}$ and \eqref{weaklowerboundY}, since the map $\psi \mapsto \varphi=\psi^{\frac{1}{1-\gamma}}$ is continuous on $\Omega_{\ge0}$, we can apply the contraction principle and obtain that $W \left( X^{\varepsilon} \in .\right)$ satisfies a large deviation weak lower bound with rate function $\bar{I}_T$.
Let us describe $\bar{I}_T(\varphi)$ when $\varphi$ is absolutely continuous and such that $I_T(\varphi)<\infty$ (where $I_T$ was defined in \eqref{e:rateFct}).
Let $\psi_t=\varphi_t^{1-\gamma}$.
On $\{\varphi = 0\}$, one has $\psi=0$ as well, while for a point $t$ in the open set $\{\varphi > 0\}$ such that $\dot{\varphi}_t$ exists, one has $\dot{\psi}_t = (1-\gamma) \frac{\dot{\varphi}_t}{\varphi_t^{\gamma}}$.
Then, noting that $I_T(\varphi)<\infty$ implies that $\frac{\dot{\varphi}_t}{\varphi_t^{\gamma}}1_{\varphi_t>0}$ is integrable on $[0,T]$, $\psi$ is also absolutely continuous on $[0,T]$ (see \cite[Corollary 3.41]{LeoniSobolev}), with derivative $\dot{\psi}_t=(1-\gamma) \frac{\dot{\varphi}_t}{\varphi_t^{\gamma}}1_{\varphi>0}$
This yields \be \bar{I}_T (\varphi) =
\mathcal{I}_{T}(\psi(\varphi)) =
\frac{1}{2\sigma^2(1-\gamma)^2}\int_{0}^{T}
\bigl( (1-\gamma) \frac{\dot{\varphi}_t}{\varphi_t^{\gamma}}
-\beta(1-\gamma)\varphi_t^{1-\gamma} \bigr)^2
1_{\varphi_t \neq 0} dt = I_T(\varphi) < \infty \ee.
If $I(\varphi)=\infty$, there is nothing to prove in \eqref{weaklowerbound}, and the claim follows.
\end{proof}

\subsubsection{Proof of Proposition \ref{p:weakLowerY}}

This section is devoted to the proof of the large deviation weak lower bound for the process $Y^{\eps}$ in \eqref{weaklowerboundY}.
While postponing some of the most technical elements to Appendix \ref{a:appendix1}, we will make use here of the following notation: for every $ h \in H, y \in \mathbb{R}$, we define $ \mathcal{S}_{y}(h)$ to be the unique solution on $[0,T]$ of the ODE
\be \label{ODEy}
\dot{\psi}_{t} = \beta (1-\gamma) \psi_{t} + \sigma (1-\gamma) \dot{h}_{t},  \quad \psi_{0}=y.
\ee
We denote $W^{\varepsilon,h}$ the measure on $\Omega$ associated to the Girsanov shift $-\frac1{\eps}\int_0^T \dot{h}_t dt$,

\begin{equation}\label{density}
\frac{ d W^{\varepsilon,h}}{dW} \left( \omega \right) = \exp \left( \frac{1}{\varepsilon} \int_{0}^{T} \dot{h}_{t}dB_{t} -\frac{1}{2\varepsilon^{2}} \int_{0}^{T} \dot{h}^{2}_{t} dt\right).
\end{equation}
An application of Girsanov's Theorem shows that $W \left( X^{\varepsilon,h} \in \cdot \right) \stackrel{ d}{=} W ^{\varepsilon,h} \left(X^{\varepsilon} \in \cdot \right)$, where $X^{\varepsilon,h}$ solves:
\begin{equation} \label{GirsanovCIRCEV}
dX^{\varepsilon,h}_{t} = b^{\varepsilon} (X^{\varepsilon,h}_{t} )dt +\sigma {|X^{\varepsilon,h}_{t}|}^{\gamma} \dot{h}_{t} dt + \varepsilon \sigma  {|X^{\varepsilon,h}_{t}|}^{\gamma} dB_{t}, \quad X^{\varepsilon,h}_{0} = \varepsilon^{\frac{1}{1-\gamma}} x.
\end{equation}
We also define the process $Y^{\varepsilon,h} := |X^{\varepsilon,h}|^{1-\gamma}$.

\begin{remark}
Note that for \eqref{GirsanovCIRCEV} there exists a weak solution, which we construct directly from a solution of \eqref{scaledCIRCEV} applying Girsanov's Theorem.
Since pathwise uniqueness holds for the couple $(b,\sigma)$, another application of the same theorem shows that pathwise uniqueness for \eqref{scaledCIRCEV} implies pathwise uniqueness for \eqref{GirsanovCIRCEV}. Therefore we can always assume that $X^{\varepsilon,h}$  solves \eqref{GirsanovCIRCEV} with the Brownian motion $B$.
\end{remark}

Two main ingredients enter in the proof of Proposition \ref{p:weakLowerY}: the convergence in law (under some conditions on $h$) of the process $Y^{\eps,h}$ to the deterministic limit $S_0(h)$ under the measure $W$ (equivalently: the weak convergence of the measure $W^{\varepsilon,h} \left( Y^{\varepsilon} \in .\right)$ to $\delta_{S_{0}(h)}$), and a lower bound for the probability $W\left(Y^{\varepsilon} \in B (\psi, R )  \right)$ depending explicitly on the relative entropy between the two measures $W^{\eps,h}$ and $W$.
This is the content of the two following lemmas.

\begin{lemma}[Convergence in law of $Y_{\cdot}^{\eps,h}$] \label{WeakConv}
Let $h \in H$ be such that
\begin{equation}\label{WeakConvAss}
(i) \ S_{0}(h)_{t}  > 0, \quad \forall t \ \in (0,T]; \qquad (ii) \ \dot{h}_{t} > k \mbox{ in a neighborhood of $0$, for some $k >0$}.
\end{equation}
Then, the process $Y^{\eps,h}$ converges in law to $S_0(h)$ under $W$, as $\eps \to 0$.
\end{lemma}

\begin{lemma}[Relative entropy bound] \label{l:entropy}
Let $ \left(\Omega,\mathcal{F} \right)$ be a probability space and $P$,$Q$ two probability measures on $\left( \Omega , \mathcal{F} \right)$ such that $dQ=FdP$. The relative entropy $H(Q|P)$ is defined as:
\begin{equation*}
H(Q|P):= \int_{\Omega} F \log(F) dP
\end{equation*}
Then, $\forall A \in \mathcal{F}$ we have:
\begin{equation} \label{e:entropy}
\log \left( \frac{P(A)}{Q(A)} \right) \geq -\frac{e^{-1} + H(Q|P)}{Q(A)}.
\end{equation}
\end{lemma}

\begin{proof}
Applying Jensen's inequality, one has
\begin{equation*}
\log \left( \frac{P(A)}{Q(A)} \right) \geq \log \left( \int_{A} F^{-1} \frac{dQ}{ Q(A)} \right) \geq 
-\frac{1}{Q(A)}\int_{A}\log( F )  dQ   \geq -\frac{1}{ Q(A)}\int_{A}(\log( F ) F)^{+} dP. 
\end{equation*}
Using the elementary fact that $\inf_{x \geq 0} x \log (x) \geq -\frac{1}{e}$:
\begin{equation*}
-\frac{1}{ Q(A)}\int_{A}(\log( F ) F)^{+} dP \geq -\frac{e^{-1} + H(Q|P)}{Q(A)},
\end{equation*}
which proves \eqref{e:entropy}.
\qedBl
\end{proof}

The relative entropy $H(W^{\varepsilon,h}|W)$ is easily computed using the martingale property of $F^{\eps,h}_t=\exp\bigl(\frac{1}{\varepsilon} \int_{0}^{t} \dot{h}_{s}dB_{s} -\frac{1}{2\varepsilon^{2}} \int_{0}^{t} \dot{h}^{2}_{s} ds \bigr)$ and It\^o isometry:
\[
\begin{aligned}
H(W^{\varepsilon,h}|W) &=
\ex \left( F_T^{\eps,h}
\left( \frac{1}{\varepsilon} \int_{0}^{T} \dot{h}_{t}dB_{t} -\frac{1}{2\varepsilon^{2}} \int_{0}^{T} \dot{h}^{2}_{t} dt\right)\right) \\
&= \ex \left( \frac{1}{\varepsilon}\int_{0}^{T}
F_t^{\eps,h}
\dot{h}_t dB_{t}
\times
\frac{1}{\varepsilon} \int_{0}^{T} \dot{h}_{t}dB_{t} 
\right)
-\frac{1}{2\varepsilon^{2}} \int_{0}^{T} \dot{h}^{2}_{t} dt 
\\
&= \frac{1}{\varepsilon^2} \int_{0}^{T}\dot{h}^{2}_{t}dt  -\frac{1}{2\varepsilon^{2}} \int_{0}^{T} \dot{h}^{2}_{t} dt,
\end{aligned}
\]
therefore
\begin{equation} \label{e:GirsEntropy}
H \left( W^{\varepsilon,h} | W  \right) = \frac{1}{2\varepsilon^{2}} \int_{0}^{T} \dot{h}^{2}_{t} dt.
\end{equation}

The proof of Lemma \ref{WeakConv} is postponed to Appendix \ref{a:appendix1}; using this lemma and Lemma \ref{l:entropy}, we can achieve here the proof of Proposition \ref{p:weakLowerY}, completing the proof of the large deviation weak lower bound for the process $X^{\eps}$.

\begin{proof}[Proof of Proposition \ref{p:weakLowerY}]
If $\mathcal{I}_{T}(\psi)=\infty$, \eqref{weaklowerboundY} is trivially true.
Then, consider $\psi \in \Omega_{\ge 0}$ such that $\mathcal{I}_{T}(\psi)<\infty$, and define $h \in H$ by setting $\dot{h}_t = \frac{\dot{\psi}_{t} - \beta(1-\gamma) \psi_{t}}{\sigma(1-\gamma)} $, so that $\mathcal{S}_{0}(h)= \psi$.

\emph{Step 1.} Assume that $h$ is such that \eqref{WeakConvAss} holds true.
An application of the relative entropy bound \eqref{e:entropy} with $P=W$, $Q=W^{\eps,h}$ yields
\begin{equation*}
\varepsilon^{2} \log \left( W \left(Y^{\varepsilon} \in B(\psi,R) \right) \right) \geq -\varepsilon^{2} \frac{\left( e^{-1} + H(W^{\varepsilon,h} | W )\right)} {W^{\varepsilon,h} \left( Y^{\varepsilon} \in B(\psi,R) \right)}
+ \varepsilon^{2} \log W^{\varepsilon,h} ( Y^{\varepsilon} \in B(\psi,R)).
\end{equation*}
Using $W^{\varepsilon,h} ( Y^{\varepsilon} \in B(\psi,R))= W^{\varepsilon} ( Y^{\varepsilon,h} \in B(\psi,R)) \to 1$ for every $R>0$ by Proposition \ref{WeakConv}, and the expression of $H(W^{\varepsilon,h} | W )$ from \eqref{e:GirsEntropy}, taking the limit as $\eps \to 0$ we obtain \eqref{weaklowerboundY}.

\emph{Step 2.} Assume now $\psi \in C^{1}([0,1])$.
Let $h$ be defined as above, and define $h^n \in H$, $n \in \mathbb{N}$, by
 \begin{equation}\label{hn}
\dot{h}^{n}_{t} := \dot{h}_t +1/n.
\end{equation} 
We claim that $\forall n \in \mathbb{N}$, ${h}^n$ satisfies \eqref{WeakConvAss}. 
Let us first prove that condition (ii) in \eqref{WeakConvAss} holds.
Observe that $\psi \geq 0$ and $\psi_0 = 0$ imply $\dot{\psi}_0 \geq 0$, hence $\dot{h}^n_0 \geq 1/n$.
By the continuity of $\dot{h}^n$, ensured by the fact that $\psi \in C^1([0,T])$, it follows that the condition $(ii)$ in \eqref{WeakConvAss} holds with, say, $k=1/(2n)$.
In order to prove condition $(i)$, we observe that the comparison principle for ODEs implies that $\forall t \in (0,T], \ \mathcal{S}_0(h^n)_{t} > \mathcal{S}_0(h)_{t} = \psi_{t} \geq 0$; condition $(i)$ is then proved.
Furthermore, by the continuity of the solution to \eqref{ODEy} with respect to the control parameter $h$, one has
\begin{equation}\label{densityarg}
\| \mathcal{S}_{0}(h^n)- \psi \|_{\infty} \rightarrow 0 \qquad \mbox{as } n \to \infty. 
\end{equation}
It follows from \eqref{densityarg} that, for any $R>0$
\begin{equation} \label{e:weakYn}
W \left( Y^{\varepsilon} \in B (\psi, R )  \right) \geq W \left( Y^{\varepsilon} \in B (\mathcal{S}_0(h^n), R/2 )  \right) 
\end{equation}
if $n$ is large enough.
In the first part of the proof, we have shown that the weak lower bound holds for $W \left( Y^{\varepsilon} \in B(\mathcal{S}_0(h^n), R/2) \right)$; then, taking the limits as $\varepsilon\to 0$ and $R\to 0$ in \eqref{e:weakYn}, one has
\begin{equation*}
\liminf_{R \rightarrow 0 } \liminf_{\varepsilon \rightarrow 0 } \varepsilon^{2} \log W \left( Y^{\varepsilon} \in B (\psi, R )  \right) \geq -\mathcal{I}_{T}(\mathcal{S}_0(h^n)) \quad \mbox{for every } n\in \mathbb{N}.
\end{equation*}
Since $\mathcal{I}_{T}(\mathcal{S}_{0} (h^n) )
= \frac12 \int_0^T (\dot h^n)^2 dt
\rightarrow
\frac12 \int_0^T (\dot h)^2 dt
= \mathcal{I}_{T}(\psi)$, the bound \eqref{weaklowerboundY} follows.
Finally, a standard density argument of $C^1([0,1])$ functions in $C([0,1])$ allows to extend the claim to any $\psi \in \Omega_{\ge 0}$ such that $\mathcal{I}_{T} (\psi)<+\infty$.
\qedBl
\end{proof}

\begin{remark} \label{r:particularLimit}
In a classical situation, the claim would be the lower bound \eqref{weaklowerbound} for a process $X^{\varepsilon}$ satisfying, say, $dX^{\eps} = b_{\eps}(X^{\eps}) + \eps \sigma(X^{\eps})dB$ with Lipschitz coefficients $\sigma$ and $b_{\eps} \to b_0$, and $X^{\eps}_0=x^{\eps} \to x$.
In this setting, fixing a control $h \in H$ and defining $X^{\eps,h}$ from $X^{\eps}$ by shifting the Brownian motion $B$ as in \eqref{GirsanovCIRCEV}, it is straightforward (in fact: an application of Gronwall's Lemma) to show that $X^{\eps,h}$ converges in law to the unique solution of the deterministic limit equation $d\varphi=b_0(\varphi)dt+\sigma(\varphi)dh, \varphi_0=x$.
In the present (degenerate) situation, the deterministic limit equation for the process $X^{\eps,h}$ (obtained setting $\eps=0$ in \eqref{GirsanovCIRCEV}) coincides with the ODE \eqref{degODE} which admits infinitely many solutions.
When circumventing this problem by passing through the transformed process $Y^{\eps,h}$, we actually show that the convergence in law of $X^{\eps,h}$ to a \emph{particular} solution $\varphi^*$ of the limiting equation is restored.
Indeed, assume as in Proposition \ref{WeakConv} that $h$ is such that the unique solution $\psi$ of the well-posed equation \eqref{ODEy} with $y=0$ is positive for every $t>0$, and $Y^{\eps,h}$ converges in law to $\psi$.
The function $\psi$ is easily computed, namely $\psi_t=\sigma(1-\gamma)e^{\beta(1-\gamma)t} \int_0^t e^{-\beta(1-\gamma)s} \dot{h}_s ds$.
By definition, one has $X^{\eps,h}=\left(Y^{\eps,h}\right)^{\frac1{1-\gamma}} \stackrel{W}{\longrightarrow} \psi^{\frac1{1-\gamma}}=:\varphi^*$.
By direct computation, $\varphi^*$ is absolutely continuous and such that $\varphi^*_0=0$ and $\dot{\varphi}^*=\beta \varphi^*+\sigma(\varphi^*)^{\gamma}\dot{h}$, hence $\varphi^*$ is \emph{a} solution to $\eqref{degODE}$; in particular,
\be \label{e:particularSol}
\varphi^{*}_{t} := e^{\beta t }
\Bigl( \sigma (1-\gamma)  \int_{0}^{t} e^{-\beta (1-\gamma) s} \dot{h}_{s} ds \Bigr)^{\frac{1}{1-\gamma}}.
\ee
Therefore, in the small noise limit, the stochastic dynamics \eqref{GirsanovCIRCEV} performs a selection among the solutions of the limiting deterministic system $\eqref{degODE}$, selecting the strictly positive one, $\varphi^*$.
This looks reasonable in light of the fact that, though converging to zero, the drift parameter $\alpha^{\eps}$ and the initial condition $x^{\eps}$ of the process remain strictly positive for all $\eps>0$.\footnote{By perturbing the initial condition and the drift in \eqref{degODE}, one can retrieve the trajectory $\varphi^*$ in \eqref{e:particularSol} as the limit as $\rho \to 0$ of the solution of the equation $d\varphi_{t} = \rho + \beta \varphi_t dt + \sigma \varphi_t^{\gamma}dh, \varphi_{0} = \rho$, for which existence and uniqueness hold.}
Figure \ref{fig1} shows the convergence of simulated trajectories of the process $X^{\eps,h}$ to $\varphi^*$ in \eqref{e:particularSol} as $\eps \to 0$, for a given choice of the control parameter $h$.
\end{remark}

\begin{remark}[Lower bound from the upper bound]
In general, the weak convergence of the controlled process $X^{\eps,h}$ can be shown exploiting the large deviation upper bound.
This goes as follows: in the notation of Remark \ref{r:particularLimit}, assume $X^{\eps}$ satisfies $dX^{\eps} = b^{\eps}(X^{\eps}) + \eps \sigma(X^{\eps})dB$ with Lipschitz coefficients,
and define $X^{\eps,h}$ from $X^{\eps}$ as in \eqref{GirsanovCIRCEV}.
Assume one has proven a large deviation upper bound analogous to \eqref{e:weakupperbound} for the process $X^{\eps,h}$, with a good rate function $I^{h}$ depending on the control parameter $h$, $I^{h} \left( \psi \right) := \frac{1}{2}\int^{T}_{0} \left( \frac{\dot{\psi}_t- b_0(\psi_t) - \sigma(\psi_t)\dot{h}_t}{\sigma(\psi_t)}\right)^2 dt$.
It is clear that $I^{h}$ admits as a unique zero the solution $\varphi(h)$ of $\dot{\psi}_t=b_0(\psi_t)+\sigma(\psi_t)\dot{h}_t$.
Using the compactness of the level sets of $I^{h}$ and the large deviation upper bound, it is easy to conclude that
\[
\lim_{\varepsilon \rightarrow 0}
W \left(X^{\varepsilon,h} \notin B(\varphi(h),R) \right)
=0
\quad \forall R>0,
\]
hence $X^{\varepsilon,h} \to \varphi(h)$ in law.
This provides a way of ``bootstrapping'' the large deviation lower bound from the upper bound (via weak convergence, together with the bound on relative entropy in Lemma \ref{l:entropy}).
When the limit ODE has several solutions, this approach is not possible anymore: in the present case, the rate function $I^{h}\left( \psi \right)= \frac{1}{2}\int^{T}_{0} \left( \frac{\dot{\psi}_t-\beta\psi_t-\psi_t^{\gamma}\dot{h}_t}{\psi^{\gamma}_t} \right)^2 \mathbb{1}_{ \left\{ \psi_t>0 \right\} }dt$ has uncountably many zeroes, corresponding to the possible solutions of the degenerate ODE \eqref{degODE}.
While one is expecting that converging subsequences of the family of measures $\{ W(X^{\eps,h} \in \cdot)\}_{\eps}$ converge to a probability distribution supported by the set of solutions, it is not obvious a priori how to restore a unique limit for $X^{\eps,h}$ (which is why we pass through the transformed process $Y^{\eps,h}$).
When uniqueness for the limiting equation is granted, such an approach remains efficient, and applies outside the Markovian framework (see \cite{ChiariniFischer} for a treatment of delayed equations. 
In the setting of \cite{ChiariniFischer}, uniqueness of solutions for the deterministic sytem is essential, and enters via their condition (H4)).
\end{remark}

\begin{figure}[t]
\centering
\includegraphics[scale=0.7]{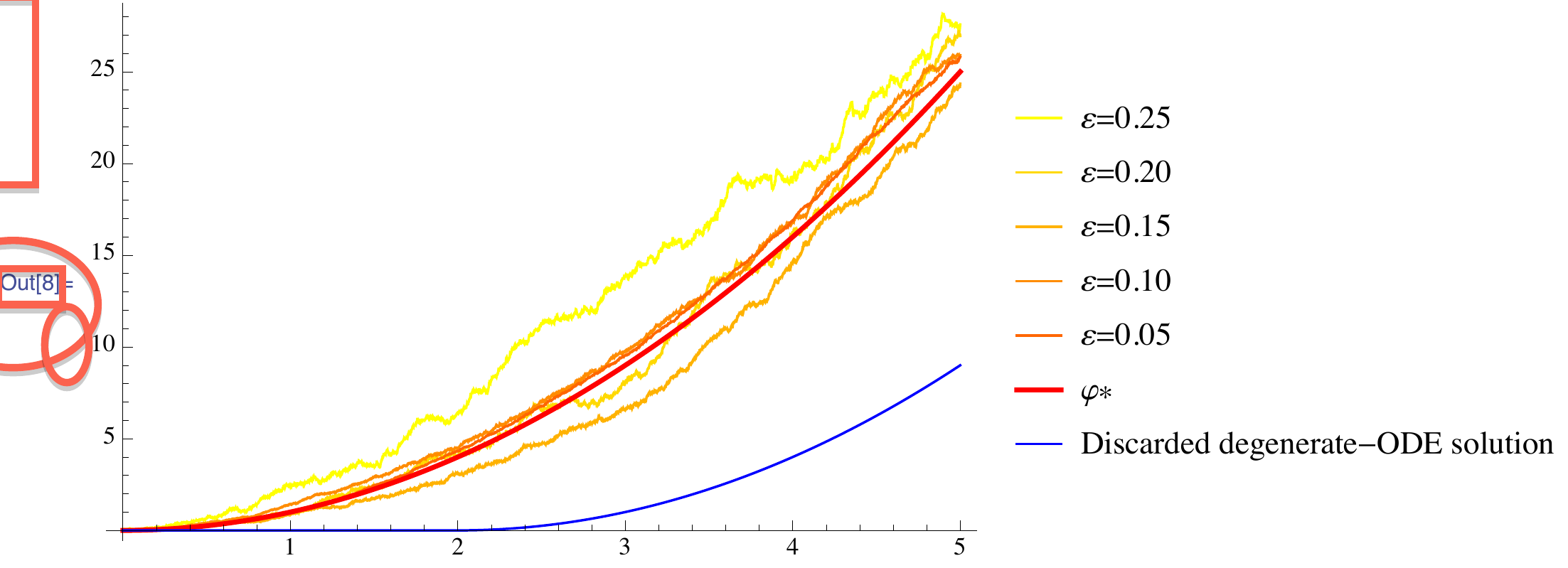}
\caption{An illustration of the convergence
of the process $X^{\eps,h}$ in \eqref{GirsanovCIRCEV} to a particular solution $\varphi^*$ of the limiting deterministic sytem.
Trajectories have been simulated for different values of the noise parameter $\varepsilon $ and $\gamma=1/2,\alpha(x)\equiv 1,\beta=0,\sigma=2, \dot{h}=1, x=0$.
}
\label{fig1}
\end{figure}

\subsection{Proof of tail estimates} \label{s:tailAsymptProofs}

In this section, we prove the asymptotic estimates that have been stated in Section \ref{s:tailAs} and that follow from Theorem \ref{main}.

\begin{proof}[Proof of Proposition \ref{CIRCEVtails}]
Setting $\eps: = R^{-(1-\gamma)}$ into \eqref{e:LD}, one has
\begin{equation*}
\limsup_{R \rightarrow + \infty } R^{-2(1-\gamma)} \log W \left( X_{T} \geq R \right)
=
\limsup_{\eps \rightarrow 0 } \eps^2 \log W \left( X^{\eps}_{T} \geq 1 \right)
\le -P
\end{equation*}
where
\begin{equation*}
\begin{aligned}
P &= 
\inf\left\{I_{T} \left( \varphi \right): \varphi_{0}=0, \varphi \ge 0, \varphi_{T} \geq 1 \right\}
\\
&= 
\inf_{y \ge 1}
\inf \left\{I_{T} \left( \varphi \right): \varphi_{0}=0, \varphi \ge 0, \varphi_{T} \geq y \right\}
=:
\inf_{y \ge 1} P(y).
\end{aligned}
\end{equation*}
Fix $y \ge 1$ and a function $\varphi$ in the admissible set of $P(y)$, such that $I_T(\varphi)<\infty$.
Set $\psi_t=\varphi^{1-\gamma}_t$.
On $\{\varphi = 0\}$, one has $\psi=0$ as well, while for a point $t$ in the open set $\{\varphi > 0\}$ such that $\dot{\varphi}_t$ exists, one has $\dot{\psi}_t = (1-\gamma) \frac{\dot{\varphi}_t}{\varphi_t^{\gamma}}$.
Then, noting that $I_T(\varphi)<\infty$ implies that $\frac{\dot{\varphi}_t}{\varphi_t^{\gamma}}1_{\varphi_t>0}$ is integrable on $[0,T]$, $\psi$ is also absolutely continuous on $[0,T]$ (see \cite[Corollary 3.41]{LeoniSobolev}).
Moreover, $I_T(\varphi)=\frac1{2\sigma^2} \int_0^T \bigl(\frac{\dot{\varphi}_t-\beta\varphi_t}{\varphi_t^{\gamma}} \bigr)^2 1_{\varphi_t>0} dt = \frac1{2\sigma^2(1-\gamma)^2}\int_0^T (\dot{\psi}_t-\beta(1-\gamma)\psi_t)^2 1_{\psi_t>0} dt$.
Noting that the inverse transformation $\varphi=\psi^{\frac1{(1-\gamma)}}$ also maps AC positive functions to AC positive functions (as $\frac1{(1-\gamma)}>1$), one has
\[
P(y) = 
\frac1{2\sigma^2(1-\gamma)^2}
\inf \left\{\int_0^T \bigl(\dot{\psi}_t-\beta(1-\gamma)\psi_t\bigr)^2 1_{\psi_t>0} dt: \psi \mbox{ is abs. cont.}, \psi_{0}=0, \psi \ge 0, \psi_{T} = y^{1-\gamma} \right\}.
\]
When $\beta=0$, the minimizer of this problem is $\psi^*_t(y)=y^{1-\gamma} t/T$.
When $\beta\neq 0$, the solution of the Euler-Lagrange equation associated with the Lagrangian $(\dot{\psi}-\beta(1-\gamma)\psi)^2$ and the boundary conditions $\psi_0=0,\psi_T=y^{1-\gamma}$ yields the minimizer
\[
\psi^*_t(y) = \frac{y^{1-\gamma}}{e^{\beta(1-\gamma)T}-e^{-\beta(1-\gamma)T}}
(e^{\beta(1-\gamma)t}-e^{-\beta(1-\gamma)t}).
\]
In both cases, $\psi^*_t(y)>0$ for all $t\in(0,T]$, and the positivity constraint in $P(y)$ can be dropped.
Using the monotonicity of $\psi^*$ w.r.t. $y$, this yields $\inf_{y \ge 1} P(y)=P(1)=\frac1{2\sigma^2(1-\gamma)^2} \int_0^T \bigl(\dot{\psi}^*_t(1)-\beta(1-\gamma)\psi^*_t(1)\bigr)^2 dt$.
An application of the large deviation lower bound \eqref{e:LD} gives
$\liminf_{R \rightarrow + \infty } R^{-2(1-\gamma)} \log W \left( X_{T} > R \right)
= \liminf_{\eps \rightarrow 0 } \eps^2 \log W \left( X^{\eps}_{T} > 1 \right) = -\inf_{y>1} P(y) = -P(1)$.
Finally, the explicit evaluation of the integral in $P(1)$ over the function $\psi^*$ yields the expression of the constant $c_T$ in \eqref{e:leadingOrder}.

Let us consider the running maximum process. Another application of the large deviation principle \eqref{e:LD} with $\eps = R^{-(1-\gamma)}$ gives
\begin{equation*}
\liminf_{R \rightarrow + \infty } R^{-2(1-\gamma)} \log W \Bigl( \sup_{t\in [0,T]} X_t > R \Bigr) \ge -\underline{c}_T
\end{equation*}
where $\underline{c}_T
\inf \bigl\{I_{T} \left( \varphi \right): \varphi_{0}=0, \varphi\ge 0, \sup_{t \in [0,T]} \varphi_{t} > 1 \bigr\}$.
Since $W\left(\sup_{t \in [0,T]} X_{t} > R\right) \geq W(X_{t} > R)$ for every $t \leq T$, one has $\underline{c}_T \leq \inf_{t \in [0,T]} c_t = c_T$, where the last identity holds for $c_t$ is a decreasing function of $t$.
On the other hand,
$\limsup_{R \rightarrow + \infty } R^{-2(1-\gamma)} \log W \bigl( \sup_{t\in [0,T]} X_t \ge R \bigr) \le -\overline{c}_T := -\inf \bigl\{I_{T} \left( \varphi \right): \varphi_{0}=0, \varphi\ge 0, \sup_{t \in [0,T]} \varphi_{t} \ge 1 \bigr\}$.
Since
\begin{align*}
\overline{c}_T &= 
\inf \Bigl\{I_T \left( \varphi \right): \varphi_{0}=0, \varphi\ge0,
\sup_{t \in [0,T]} \varphi_{t}=1, \varphi_{t} \geq 0 \Bigr\}
\\
&\geq
\inf_{t \in [0,T]}
\inf \{ I_{t}(\phi): \phi \mbox{ is abs. cont. on $[0,t]$}, \phi_0=0, \phi\ge 0, \phi_{t}=1 \}
\\
&= \inf_{t \in [0,T]} c_t = c_T
\end{align*}
one has $\underline{c}_T = \overline{c}_T = c_T$, and the claim is proved.
\qedBl
\end{proof}
\medskip

As addressed in Section \ref{s:tailAs}, Theorem \ref{main} can also be used to obtain the leading-order asymptotics for the distribution of the time average of the process.
Such a result can be used to derive the leading-order behavior of the implied volatility of Asian options $\esp\bigl[\bigl(\frac{1}{T}\int_{0}^{T} X_{t}dt-K\bigr)^+\bigr]$ for large strike $K$.

\begin{proposition} \label{LaplaceAsian}
Estimate \eqref{e:tailAsAsian} in Theorem \ref{t:tailIntro} holds with $\nu_T>0$. 
When $\gamma=1/2$, the constant $\nu_T$ is given by
\begin{equation} \label{e:leadingAsian}
\nu_T =
\begin{cases}
\frac1{2\sigma^2} \Bigl(T\beta^2 + \frac{4 \omega^2}{T}\Bigr) 

& \text{ if $T \beta/2 < 1$ }
\\
\frac1{2\sigma^2} \Bigl(T\beta^2 - \frac{4\omega^2}{T}\Bigr) 
& \text{ if $T \beta/2 \geq 1$ } \end{cases} 
\end{equation}
where
\begin{equation} \label{e:omegaDef}
\omega = \left\{
\begin{array}{l l}
\mbox{the $\omega \in (0,\pi)$ such that
$\omega \cos\omega = T\beta/2 \sin(\omega)$}
& \text{ if $T \beta/2 < 1$}
\\
0 & \text{ if $T \beta/2= 1$}
\\
\mbox{the $\omega \in (0,\infty)$ such that
$\omega \cosh(\omega) = T\beta/2 \sinh(\omega)$}
& \text{ if $T \beta(1-\gamma)  \geq 1$.}
\end{array} \right.
\end{equation}
\end{proposition}

\begin{remark}
\emph{Following the lines of the proof of Proposition \ref{LaplaceAsian}, one can prove the analogous asymptotic relation for a general time-average functional $\int_0^T X_t \mu(dt)$, where $\mu$ is a bounded signed measure on $[0,T]$.
One gets
\begin{equation*}
W \left( \int_{0}^{T} X_{t} \mu(dt)  \geq R \right) = e^{-R^{2(1-\gamma)}(\mathcal{V}_T+\psi(R))} \qquad \mbox{as } R \to \infty,
\end{equation*}
where $\mathcal{V}_T$ is characterised by the variational formula $\mathcal{V}_T := \inf \Bigl\{ I_{T} ( \varphi ): \int_{0}^{T} \varphi_{t} \mu(dt) \geq 1, \varphi_{t} \geq 0, \forall t \in [0,T] \Bigr\}$.}
\end{remark}

\begin{proof}[Proof of Proposition \ref{LaplaceAsian}]
An application of th large deviation principle \eqref{e:LD} with $\eps := R^{-(1-\gamma)}$ yields $\limsup_{R \rightarrow + \infty } R^{-2(1-\gamma)} \log W \bigl( \frac1T \int_0^T X_t dt \ge R \bigr)=\limsup_{\eps^2 \rightarrow 0} \eps^{2} \log W \bigl(\frac1T \int_0^T X^{\eps}_t dt \ge 1 \bigr) \le -\nu_T$, with $\nu_T = \inf \{ I_{T} \left( \varphi \right) : \varphi_{0}=0, \varphi \ge 0, \frac{1}{T} \int_{0}^{T} \varphi_{t}dt\ge1 \}$.
Proceeding as in the proof of Proposition \ref{CIRCEVtails}, and in particular exploiting the endomorphism of $AC([0,T], \R_+)$ $\varphi \to \psi=\varphi^{1-\gamma}$ together with the chain rule $\dot{\psi} = \dot{\varphi}/{\varphi^{\gamma}}1_{\varphi>0}$, one has
\begin{align*}
\nu_T &=
\inf \biggl\{ I_{T} \left( \varphi \right) : \varphi_{0}=0, \varphi \ge 0, \frac{1}{T} \int_{0}^{T} \varphi_{t}dt \ge 1 \biggr\}
\\
&=
\frac{T}{2\sigma^{2}(1-\gamma)^{2}}  \inf 
\biggl\{
\int_{0}^{1} \left( \dot{\psi}_{Tt} - \beta (1-\gamma) \psi_{Tt} \right)^{2} dt: \psi_{0}=0, \psi \ge 0, \int_{0}^{1} \psi_{Tt}^{1/(1-\gamma)} dt\ge 1 \biggr\}
\\
&=
\frac{1}{2T\sigma^{2}(1-\gamma)^{2}}
\inf \biggl\{
\int_{0}^{1} \Bigr( \frac{d}{dt}{(\psi_{Tt})} - T\beta (1-\gamma ) \psi_{Tt} \Bigr)^2 dt: \psi_{0}=0, \psi \ge 0, \int_{0}^{1} \psi_{Tt}^{1/(1-\gamma)} dt \ge1
\biggr\}
\\
&=
\inf_{\eta \ge 1}
\frac{1}{2T\sigma^{2}(1-\gamma)^{2}}
\inf \bigg\{ \int_{0}^{1} \Bigl( \dot{\phi}_{t} - T\beta (1-\gamma) \phi_{t}   \Bigr)^{2} dt: \phi_{0}=0, \phi \ge 0, \int_{0}^{1} \phi_{t}^{1/(1-\gamma)} dt=\eta
\Bigr\}
=:\inf_{\eta \ge 1} J(\eta).
\end{align*}
When $\gamma=1/2$, the latter variational problem was studied in \cite[Exercise 2.1.13]{DEUSCHEL}.
The explicit solution for $J$ provides the expression of the constant $\nu_T=\inf_{\eta \ge 1} J(\eta)=J(1)$ given in \eqref{e:leadingAsian}.
The large deviation lower bound yields $\liminf_{R \rightarrow + \infty } R^{-2(1-\gamma)} \log W \bigl( \frac1T \int_0^T X_t dt > R \bigr)=\liminf_{\eps^2 \rightarrow 0} \eps^{2} \log W \bigl(\frac1T \int_0^T X^{\eps}_t dt > 1 \bigr) \ge -J(1)=\eta_T$, and the claim is proved.
\qedBl
\end{proof}
\medskip

\textbf{Consistency check with the explicit formulae for the integrated CIR process}.
Let us consider the case $\gamma=1/2$, and compare Proposition \ref{LaplaceAsian} with the moment explosion of the integrated CIR process, corresponding to $\alpha(x)\equiv\alpha\ge0$ in condition (H2).
We focus on the (common) case of a mean-reverting drift, i.e. $\beta < 0$; computations for $\beta > 0$ are similar.
Estimate \eqref{e:tailAsAsian} establishes that $\frac1T \int_0^T X_t dt$ has finite exponential moments up to order $\nu_T$: more precisely,
\be \label{e:criticalExp}
u^* := \sup\{u > 0: \esp\Bigl[\exp\Bigl(\frac uT \int_0^T X_t dt\Bigr)\Bigr]<\infty \} = \sup\{ \nu>0: \Prob\Bigl(\frac1T \int_0^T X_t dt>x\Bigr) = O(e^{-\nu x}) \mbox{ as } x \to \infty \} = \nu_T
\ee
(for the central identity, see for example \cite[Section 4]{GulForm}); in other words, $\nu_T$ is the positive critical exponent of $\frac1T \int_0^T X_t dt$.
Critical exponents for integrated CIR have been assessed by \cite{Duf,AP,KellRes} relying (essentially) on the affine structure of the process.
It is typical to obtain $u^*$ by inverting an explicit explosion time: following \cite[Corollary 3.3]{AP}, $\esp[\exp(\frac uT \int_0^T X_t dt)]$ is always finite if $u\le T\beta^2/(2\sigma^2)$, and if $u > T\beta^2/(2\sigma^2)$, the expectation is finite for $T<T^*(u)$ and infinite for $T>T^*(u)$, where $T^*$ reads
\[
T^*(u) = 2 \frac{\pi+\arctan\left(\frac{\gamma(u)}{\beta}\right)}{\gamma(u)},
\]
where $\gamma(u)=\sqrt{2\sigma^2\frac uT - \beta^2}$.
Fixing $T$ and using the monotonicity of $T^*$, this means that the expectation becomes infinite for $u>u^*$ with $u^*$ the solution to
\be \label{e:criticalExpRoot}
\pi+\arctan\left(\frac{\gamma(u)}{\beta}\right) = \frac T2 \gamma(u)
\ee
As an equation in $\gamma$, it is easy to see that \eqref{e:criticalExpRoot} has a unique root $\gamma^*$ on $\R^+$ such that $\frac T2 \gamma^* \in (\frac\pi 2,\pi)$.
From the definition of $\gamma$,
\[
u^* = \frac 1{2\sigma^2} (T \beta^2 + T (\gamma^*)^2)
= \frac 1{2\sigma^2} \Bigl(T \beta^2 + \frac4T \Bigl(\frac{T\gamma^*}2\Bigr)^2\Bigr)
= \frac 1{2\sigma^2} \Bigl(T \beta^2 + \frac4T (\omega^*)^2\Bigr)
\]
setting $\omega^*=\frac{T\gamma^*}2$. From \eqref{e:criticalExpRoot}, $\omega^*$ is the unique solution to $\omega=\pi+\arctan\left(\frac{2\omega}{T\beta}\right)$, which is equivalent to $\tan(\omega)=\frac{2\omega}{T\beta}$ together with $\omega \in (\frac\pi 2,\pi)$: one sees that this definition coincides with the one for $\omega$ in \eqref{e:omegaDef} (noticing we are in the first case when $\beta < 0$).

\appendix

\section{Appendix} \label{a:appendix1}

We complete the proof of Proposition \eqref{expmoments} here.

\begin{proof}[Proof of Proposition \ref{expmoments}]
Let us define an auxiliary process $\overline{X}$ by
\begin{equation*}
d\overline{X}_{t} =  |\alpha|_{\infty}dt +\sigma \exp(-(1-\gamma) |\beta| t ) \overline{X}_{t}^{\gamma}dB_{t}, \quad \overline{X}_{0}=x;
\end{equation*}
after a simple application of the product rule, one has that the process $Z_{t} := \exp(|\beta| t)\overline{X}_{t} $ is a solution to
\begin{equation*}
dZ_{t} = \bigl(|\alpha|_{\infty} \exp(|\beta| t) + |\beta| Z_{t}\bigr) dt + \sigma Z_{t}^{\gamma}dB_{t}, \quad Z_{0} = x.
\end{equation*}
Since $|\alpha|_{\infty} \exp(|\beta| t) \ge |\alpha|_{\infty}$, an application of the comparison principle for SDE's \cite[Proposition 5.2.18]{KS} yields $Z_{t} \geq \tilde{X}_{t}$, for all $t \ge 0$.
Therefore, if $\overline{X}^{2(1-\gamma)}$ admits (some) exponential moments, so does $Z^{2(1-\gamma)}_{t}$ and by comparison $\tilde{X}^{2(1-\gamma)}_{t}$.
In this sense, the process $\overline{X}$ is not covered by Proposition 3.3 in \cite{DEMARCO}, since the latter deals with the case of a diffusion coefficient that does not depend on time (see \cite[Eq. (3.1)]{DEMARCO}); nonetheless, the essential condition that \cite[Prop 3.3]{DEMARCO} relies on is the presence of a non-strictly positive slope coefficient, say $b$ in the drift term $a+bX$ (cf. \cite[Eq. (3.3)]{DEMARCO}).
Since this is the case for the process $\overline{X}$ (which has zero slope coefficient $b$), it is straightforward to extend the proof to the present setting: in particular, in the spirit of Lamperti's change-of-variable argument, one still defines the function $\varphi(x)=\int_0^{x}\frac1{\sigma x^{\gamma}}=\frac1{\sigma(1-\gamma)}x^{1-\gamma}$ and studies the process $\tilde{\varphi}(X_t)$, where the function $\tilde{\varphi}$ is a modification of $\varphi$ identically null around zero. 
It\^o's formula shows that $\tilde{\varphi}(X_t)$ is an It\^o process with bounded quadratic variation and a bounded drift term; the existence of quadratic exponential moments for $\tilde{\varphi}(X_t)$, then, is a consequence of Dubins--Schwarz time-change argument and Fernique's theorem.
As a consequence, there exist $c',C > 0$ such that $\sup_{t \le T} \esp[\exp(c' \overline{X}^{2(1-\gamma)}_{t})] \le C$; it follows $\sup_{t \le T} \esp[\exp(c \tilde{X}^{2(1-\gamma)}_{t})] \le \sup_{t \le T} \esp[\exp(c Z^{2(1-\gamma)}_{t})] \le C$ with $c:=c' \exp(-2|\beta|(1-\gamma)T)$, and the claim is proved.
\qedBl
\end{proof} 
\medskip

We report the statement given in \cite[Chap 2, Thm 2.13]{STROOCKVARADAHAN}.

\begin{lemma}[Garsia-Rodemich-Rumsey's Lemma] \label{l:GarsiaLem}
Let $p$ and $\Psi$ be continuous, strictly increasing functions  on $[0, +\infty)$ such that $p(0) = \Psi(0)=0$ and $\lim_{t \rightarrow + \infty  } \Psi(t)= + \infty $. If $\omega\in \Omega $ is such that: 
\begin{equation}
\int_{0}^{T} \int_{0}^{T} \Psi \left( \frac{|\omega_{t} - \omega_{s} |}{p(|t-s|)}\right) ds dt \leq K,
\end{equation}
then
\begin{equation}
|\omega_{t} - \omega_{s} | \leq 8 \int_{0}^{|t-s|} \Psi^{-1} \left(  \frac{4K}{u^2}\right)dp(u).
\end{equation}
\end{lemma}

Lemma \ref{l:GarsiaLem} allows us to prove Proposition \ref{GRRI}:

\begin{proof}[Proof of Proposition \ref{GRRI}]
Assume that \eqref{e:Garsia} holds true with the left hand side replaced by $K>0$.
Applying Lemma \ref{l:GarsiaLem} with the choice of functions $\Psi(y)=\exp(\varepsilon^{-2} y)-1, p(y) = \sqrt{y}$, one has for all $s,t$
\[
\begin{aligned}
\left| \omega_{t}- \omega_{s} \right| \leq 8\int_{0}^{|t-s|} \Psi^{-1} \left(  \frac{4K}{u^2} \right) dp(u)
&= 
8\varepsilon^{2}\int_{0}^{|t-s|} \log \left(  \frac{4K}{u^2} +1 \right) dp(u)
\\
&\leq 
8 \varepsilon^{2} \left[ \int_{0}^{|t-s|}  \log \left(  4K + T^2 \right) dp(u)
+ \int_{0}^{|t-s|} \log \left( u^{-2} \right) dp(u) \right] 
\\
&\leq
8 \varepsilon^{2}\left[ \sqrt{|t-s|} \log \left(4K+T^{2} \right) 
+ \sqrt{|t-s|} \left(4-2\log \left(|t-s|\right) \right) \right].
\end{aligned}
\]
Dividing on both sides by $(t-s)^{\eta}$ and taking suprema we obtain
\begin{equation*}
\left\| \omega \right\|_{\eta} \leq 8 \varepsilon^{2} \left(\log \left(4K+T^{2} \right)T^{1/2-\eta} +4T^{1/2-\eta} +  K_{\eta} \right).
\end{equation*}
Since the right hand side in the last estimate is $K^{-1}_{\varepsilon, \eta } \left( K \right)$, \eqref{e:Garsia} yields \eqref{e:GarsiaHoldNorm}.
\qedBl
\end{proof}

Finally, we prove Lemma \ref{WeakConv}.
\begin{proof}[Proof of Lemma \ref{WeakConv}]
Denote $T^{\varepsilon}$ the stopping time
\begin{equation}\label{stopping}
T^{\varepsilon} (\omega) = \inf \left\{ t \geq 0 :  \omega_{t} \leq  \frac12 \varepsilon x^{1-\gamma} \right\}.
\end{equation}
We can apply It\^{o} formula to the function $f(x) = x^{1-\gamma}$ up to time $T^{\varepsilon}(Y^{\varepsilon,h})$, and obtain
\begin{equation}\label{yepsh}
Y^{\varepsilon,h}_{t}  - \varepsilon x^{1-\gamma}=  \int_{0}^{t} \tilde{b}^{\varepsilon} (Y^{\varepsilon,h}_{s})ds + \sigma(1-\gamma) h_{t} + \varepsilon \sigma(1-\gamma) B_t,
\quad
\forall \: t \leq T^{\varepsilon}(Y^{\varepsilon,h}),
\quad
a.s.
\end{equation}
where $\tilde{b}^{\varepsilon}$ is given by
\begin{equation}\label{singulardrift}
\tilde{b}_{\varepsilon}(y) := (1-\gamma)\varepsilon^{\frac{1}{1-\gamma}}\alpha(\varepsilon^{-\frac{1}{(1-\gamma)}}y^{\frac{1}{(1-\gamma)}})\frac1{y^{\frac{\gamma}{1-\gamma}}} - \frac{\sigma^{2} \gamma(1-\gamma)}{2}\varepsilon^{2} \frac1 y + \beta(1-\gamma) y
\end{equation}
We need to prove
\begin{equation} \label{lawcon}
\lim_{\varepsilon \rightarrow 0 } W \left( \sup_{t \in [0,T]} | Y^{\varepsilon,h}_{t}- \mathcal{S}_{0}(h)_{t}| \leq R \right) = 1
\qquad \forall R>0.
\end{equation}
In order to simplify the notation, there is no ambiguity in writing $Y$ instead of $Y^{\varepsilon,h}$ inside this proof.

\emph{Step 1.}
We first prove \eqref{lawcon} under the assumption
\begin{equation}\label{uniform}
 k:=\inf_{t \in [0,T]}\dot{h}_{t} >0
 \end{equation}
Let us fist show that 
\begin{equation} \label{singularity}
\lim_{ \varepsilon \rightarrow 0}  W \left( T^{\varepsilon} \left( Y^{\varepsilon,h} \right) \leq T \right) =0
\end{equation}
A direct computation shows that there exist a constant $c>0$ depending on $x,\sigma,\alpha(\cdot)$ such that: 
\begin{equation}\label{inf}
\inf_{y \geq \frac12 \varepsilon x^{1-\gamma}}
\Bigl\{
\tilde{b}^{\varepsilon} (y)-\beta(1-\gamma) y
\Bigr\}
\geq - c \varepsilon.
\end{equation}
Define $(Z_{t})_{t \in [0,T]}$
by
\begin{equation}
 Z_{t} = \varepsilon x^{1-\gamma} + \left( -c\varepsilon +\sigma (1-\gamma) k \right) t + \beta(1-\gamma) \int_{0}^{t} Z _{s}ds + \varepsilon \sigma(1-\gamma) B_{t}     
 \end{equation}
Using \eqref{inf}, it follows from the comparison principle for SDEs that
\begin{equation}\label{comparison}
\quad Y_{t} \geq Z_{t} \quad  \forall \: t \leq T^{\varepsilon}(Y),
\quad a.s.
\end{equation}
We claim that
\be \label{zcrossing}
W \left( T^{\varepsilon} \left( Z\right) \leq T \right) \to 0
\ee
holds true.
Since $W \left( T^{\varepsilon} \left( Y \right) \leq T \right) \leq W \left( T^{\varepsilon} \left( Z\right) \leq T \right)$ by \eqref{comparison}, then \eqref{singularity} holds.
%
We prove \eqref{zcrossing} later on.
Now, it follows from the definition of $S_0(h)_t$ and an application of Gronwall's Lemma that 
\begin{equation*}
 |Y_{t}  -\mathcal{S}_{0}(h)_{t}|  \leq \varepsilon
\biggl(
c+ \sigma(1-\gamma) \sup_{t \in [0,T]}| B_{t}|
\biggr)
e^{|\beta|(1-\gamma)T}
=:\Theta_T
\qquad \forall t\leq T^{\varepsilon} \left( Y \right),
\end{equation*}
therefore, for any $R >0$ and $\varepsilon$ small enough
\begin{align*}
W \biggl( \sup_{t \in [0,T^{\eps}] } |Y_{t}  -\mathcal{S}_{0}(h)_{t}|  \leq R \biggr)
& \geq
W \biggl( \biggl\{ \sup_{t \in [0,T^{\eps}(Y)] }|Y_{t} -\mathcal{S}_{0} (h)_{t}| \leq 
\Theta_T
\biggr\}
\cap
\left\{ \Theta^{\eps}_T
\leq R \right\}
\biggr)
\\
&\geq
W \left( \left\{ T^{\varepsilon} (Y) \geq T \right\}
\cap  \left\{
\Theta^{\eps}_T
\leq R \right\} \right).
\end{align*} 
Since both the events in the right hand side of the last inequality have probability converging to $1$, \eqref{lawcon} follows, and Lemma \ref{WeakConv} is proved under condition \eqref{uniform}.

\emph{Step 2.}
We assume that \eqref{uniform} holds only on the time interval $[0,\rho]$, that is $\dot{h}_{t} \geq k$ for every $t \leq \rho$, for some $k, \rho>0$.
Repeating the  argument of Step 1 with $T= \rho$, we have
\begin{equation}\label{escape}
\lim_{\varepsilon \rightarrow 0} W\biggl(\sup_{t \in [0,\rho]}  | Y_{t}-S_{0}(h)_{t} | \leq R^{'} \biggr) =1, \quad \forall R^{'}>0
\end{equation}    

We apply estimate \eqref{escape} together with a localization argument.
Define a time-shift operator $\tau_{\rho} \omega$, for every $\omega \in \Omega$, by $ (\tau_{ \rho } \omega)_{t} = \omega_{\rho+t}$ for all $t \in [0, T-\rho]$.
For any fixed $y>0$, denote $X^{y,\rho}$ the strong solution of the SDE:
\begin{equation*}
  X^{y,\rho}_{t} = y^{\frac{1}{(1-\gamma)}} + \int_{0}^{t} b^{\varepsilon}(X^{y,\rho}_{s}) + \sigma |X^{y,\rho}_{s}|^{\gamma} \dot{h}_{\rho+s} ds + \varepsilon \sigma \int_{0}^{t} |X^{y,\rho}_{s}|^{\gamma} dB_{s}
\end{equation*}
and set
\begin{equation*}
Y^{y,\rho}:= (X^{y,\rho})^{1-\gamma}.
\end{equation*}
Note that $Y^{y,\rho}$ is well defined since $X^{y,\rho} \geq 0$ for all $t \in [0,T]$, $W$-almost surely. 
If $h=0$ the non negativity of the trajectories of $X^{y,\rho}$ follows from an application Proposition 3.1 in \cite{DEMARCO}
and extends to $h \in H $ by an application of the Girsanov theorem.
By definition of $Y$ and $Y^{y,\rho}$, the Markov property yields
\begin{equation*}
 \mathbb{E} ( f(\tau_{\rho}Y) | \mathcal{F}_{\rho}) = \mathbb{E} ( f(Y^{Y_{\rho},\rho}) )
\end{equation*}
By the continuity of the map $(h,y) \mapsto \mathcal{S}_{y}(h)$ we can choose $R'>0$ such that
\begin{equation}\label{flowcont}
 \sup_{y \in B(S_{0}(h)_{\rho},R')} \sup_{t \in [0,T-\rho]}  | \mathcal{S}_{y}(\tau_{\rho} h)_{t} -  \mathcal{S}_{\mathcal{S}_{0}(h)_{\rho}}(\tau_{\rho}h)_{t} | \leq \frac{R}{2} 
\end{equation}
Therefore, using \eqref{flowcont} the following inclusion of events holds (assume w.lo.g $R' \leq \frac{R}{2}$):
\begin{equation*}
\biggl\{ \sup_{t \in [0,T]} |Y_{t} - \mathcal{S}_{0} (h)_{t} | \leq R
\biggr\}
\supseteq
\biggl\{ \sup_{[0,\rho]}|Y_{t} - \mathcal{S}_{0}(h)_{t} | \leq R'  \biggr\}
\cap
\biggl\{ \sup_{t\in[0,T-\rho]} |\tau_{\rho}(Y)_{t} -  \mathcal{S}_{Y_{\rho}}(\tau_{\rho}h)_{t} | \leq \frac{R}{2}
\biggr\}   
\end{equation*} 	
Applying the Markov property
\begin{align}\label{Markov}
W \biggl( \sup_{t \in [0,T]} |Y_{t} - \mathcal{S}_{0}(h)_{t} |  \leq R \biggr) & \geq \mathbb{E} \biggl( \mathbb{1}_{ \{ \sup_{t \in [0,\rho]}|Y_{t} - \mathcal{S}_{0}(h)_{t} | \leq R' \}}  W \biggl( \sup_{t \in [0,T-\rho]} |Y^{ Y_{\rho}, \rho }_{t} -  S_{Y_{\rho}}(\tau_{\rho}h)_{t} | \leq \frac{R}{2}  \biggr) \biggr) \nonumber \\
 & \geq W \biggl( \sup_{t \in [0,\rho]}|Y_{t} - \mathcal{S}_{0}(h)_{t} | \leq R' \biggr) \inf_{y \in B(\mathcal{S}_{0} (h)_{\rho} ,R') } W \biggl( \sup_{t \in [0,T-\rho]} | Y^{y,\rho}_{t} -  \mathcal{S}_{y}(\tau_{ \rho}h)_{t}|  \leq \frac{R}{2} \biggr)
\end{align}
We want to show that
\begin{equation}\label{classicallocal}
 \lim_{\varepsilon \rightarrow 0 }\inf_{ y \in B \left( S_{0} \left(h \right)_{\rho},R' \right) } W \left( \sup_{t \in [0,T-\rho]} | Y^{y,\rho}_{t} -   \mathcal{S}_{y} (\tau_{ \rho } h)_{t} |  \leq \frac R2 \right) = 1 
\end{equation}
It follows from the hypothesis $\mathcal{S}_{0}(h)_{t} >0  \ \forall t>0$ and the continuity of the map $(y,h) \mapsto \mathcal{S}_{y}(h)$ that, if $R',R$ are small enough
\begin{equation}\label{away} 
y^{*}:=\inf_{y \in B\left( \mathcal{S}_{0} \left(h \right)_{\rho},R' \right)} \inf_{t \in [0,T-\rho]} \mathcal{S}_{y} ( \tau_{\rho} h )_{t} -\frac{R}{2} >0.
\end{equation}
Define $U^{y,\rho}$
as the unique strong solution of the SDE:
\begin{equation*}
U^{y,\rho}_{t}   = y + \int_{0}^{t} \bigl( \tilde{b}^{\varepsilon}_{u} (U^{y,\rho}_{s}) + \sigma(1-\gamma) \dot{h}_{s+\rho} \bigr)ds
+ \varepsilon \sigma(1-\gamma) B_{t},
\end{equation*}
where
\begin{equation*}
\tilde{b}^{\varepsilon}_{u}(y) = \begin{cases} 
\tilde{b}^{\varepsilon}(y)  & \text{if $ y \geq y^*  $ }
\\
 \beta(1-\gamma)y + (1-\gamma)\varepsilon^{\frac{1}{1-\gamma}}\alpha(\varepsilon^{-\frac{1}{(1-\gamma)}}(y^*)^{\frac{1}{(1-\gamma)}})
\frac1{(y^*)^{\frac{\gamma}{1-\gamma}}} - \frac{\sigma^{2} \gamma(1-\gamma)}{2}\varepsilon^{2} \frac1{y^{*}} & \text{if $  y < y^{*}$.}
\end{cases}
\end{equation*}
Then one has
\begin{equation}
W \biggl( \sup_{t \in [0,T - \rho ]} | Y^{y,\rho}_{t} -  S_{y}(\tau_{ \rho }h)_{t} |  \leq \frac{R}{2} \biggr)
=
W \biggl( \sup_{t \in [0, T-\rho]} | U^{y,\rho}_{t} - S_{y}(\tau_{\rho}h )_{t}| \leq \frac{R}{2} \biggr).
\end{equation}
Now observing that $\tilde{b}^{u}_{\varepsilon} $ is globally Lipschitz continuous $\forall \varepsilon>0$ and $C^{\varepsilon} := \sup_{y \in \mathbb{R}}|\tilde{b}^{\varepsilon}_{u} (y) -\beta(1-\gamma)y| \rightarrow 0$, an application of Gronwall's lemma gives
\begin{equation}\label{finest}
\mathbb{E} \biggl( \sup_{t \in [0,T-\rho]} \left| U^{y,\rho}_{t} - \mathcal{S}_{y}(\tau_{\rho} h)_{t} \right| \biggr)
\leq  (C^{\varepsilon}T + 2\varepsilon \sigma(1-\gamma)\sqrt{T} )\exp( |\beta (1-\gamma)| T).
\end{equation}
By letting $\varepsilon \rightarrow 0$  and applying the Markov inequality, observing that the right hand side of \eqref{finest} does not depend on $y$, we have proven \eqref{classicallocal}.
By letting $\varepsilon \rightarrow 0$ in \eqref{Markov} and applying \eqref{escape} and \eqref{classicallocal}, the proof of Lemma \ref{WeakConv} is complete. 
\qedBl
\end{proof}

\emph{Proof of \eqref{zcrossing}}.
Observe that $\tilde{Z}:=\frac{1}{\varepsilon}Z$ is an Ornstein-Uhlenbeck process,
\begin{equation}
\tilde{Z}_t = x^{1-\gamma} + \mu_{\varepsilon} t + \beta(1-\gamma)\int_{0}^{t}\tilde{Z}_{s}ds + \sigma(1-\gamma) B_t
\end{equation}
where $\mu_{\eps}:= \frac{1}{\varepsilon}(-c \varepsilon + \sigma(1-\gamma)k) = -c + \frac{\sigma(1-\gamma)k}{\eps}$.
It is immediate by the definition of $\tilde{Z}$ that $W \left( T^{\varepsilon}(Z) \leq T \right) = W \left(\inf_{t \in [0,T]} \tilde{Z} \leq \frac{x^{1-\gamma}}{2} \right)$.
The explicit representation of $\tilde{Z}$ reads
\begin{equation}
\tilde{Z}_t := x^{1-\gamma}e^{\beta(1-\gamma) t}
+ f_{\eps}(t)
+ \sigma(1-\gamma)\exp(\beta(1-\gamma) t) \int_{0}^{t} \exp(- \beta(1-\gamma) s) dB_s
\end{equation}
with $f_{\eps}(t) = -\frac{\mu_{\varepsilon} (1-\exp(\beta(1-\gamma) t) )}{\beta(1-\gamma)}$.
Consider a deterministic time $\tau_{\varepsilon}$ with $\tau_{\eps} \to 0$ as $\eps \to 0$, to be chosen precisely later on.
Noting that $f_{\eps}$ is a decreasing function, for $\tau_{\varepsilon} \le t \le T$ one has
\begin{equation}
\tilde{Z}_t \geq  
f_{\eps}(\tau_{\eps})
- \sigma(1-\gamma)
\Bigl| \int_{0}^{t} \exp(-\beta(1-\gamma) s ) dB_s
\Bigr|;
\end{equation}
hence, using Markov's inequality and Doob's inequality
\[
\begin{aligned}
W \left( \inf_{t \in [\tau_{\varepsilon},T]} \tilde{Z}_t  \leq x^{1-\gamma}/2 \right)
&\leq
W \left( \sup_{t \in [\tau_{\varepsilon},T]}
\sigma(1-\gamma)
\Bigl| \int_{0}^{t} \exp(-\beta(1-\gamma) s) dB_s \Bigr|
\geq
f_{\eps}(\tau_{\eps}) - x^{1-\gamma}/2 \right)
\\ 
&\leq C \sigma(1-\gamma)
\left( f_{\eps}(\tau_{\eps}) - x^{1-\gamma}/2  \right)^{-1} \left(\int_{0}^{T}\exp(-2\beta(1-\gamma) s)ds\right)^{\frac{1}{2}}.
\end{aligned}
\]
Now, the choice $\tau_{\eps} = \sqrt{\eps}$ gives $f_{\eps}(\tau_{\eps}) \sim \mu_{\eps} \tau_{\eps} \to \infty$ as $\eps \to 0$, so that $\left( f_{\eps}(\tau_{\eps}) - x^{1-\gamma}/2 \right)^{-1} \to 0$.
On the other hand, $\inf_{t \in [0,\tau_{\varepsilon}]} \tilde{Z}_t \to x^{1-\gamma}$ a.s. as $\eps \to 0$, hence $W \Bigl( \inf_{t \in [0,\tau_{\varepsilon}]} \tilde{Z}_t \leq x/2 \Bigr) \to 0$ as $\eps \to 0$, and the claim is proven.

\bibliographystyle{plain}
\bibliography{ReferencesConfortiDMDeu}

\end{document}